\documentclass[a4paper]{amsart}

\usepackage{amssymb}
\usepackage{amsmath,amsthm}
\usepackage{enumerate,paralist}
\usepackage{amstext}
\usepackage{psfrag}
\usepackage[dvips]{epsfig}
\usepackage[english]{babel}
\usepackage{color}

\theoremstyle{plain}
\newtheorem{theorem}{Theorem}
\newtheorem{corollary}{Corollary}
\newtheorem{lemma}{Lemma}
\newtheorem{proposition}{Proposition}
\theoremstyle{definition}

\newtheorem{example}{Example}
\newtheorem{remark}{Remark}
\newtheorem{assumption}{Assumption}

\numberwithin{theorem}{section}
\numberwithin{corollary}{section}
\numberwithin{lemma}{section}
\numberwithin{definition}{section}
\numberwithin{example}{section}
\numberwithin{remark}{section}
\numberwithin{proposition}{section}
\numberwithin{assumption}{section}

\usepackage{color}
\definecolor{CB}{rgb}{0.1,0.5,0.1}
\definecolor{YK}{rgb}{0.1,0.2,0.7}
\definecolor{H}{rgb}{0.7,0.1,0.2}
\definecolor{MY}{rgb}{0.5,0,0.45}

\newcommand{\sign}{\textnormal{sign}}
\newcommand{\Leb}{\textnormal{Leb}}

\renewcommand{\Re}{\ensuremath{{\rm Re\,}}}

\renewcommand{\leq}{\leqslant}
\renewcommand{\geq}{\geqslant}

\newcommand{\ffi}{\varphi}
\newcommand{\eps}{\varepsilon}

\renewcommand{\Re}{\ensuremath{{\rm Re\,}}}

\renewcommand{\leq}{\leqslant}
\renewcommand{\geq}{\geqslant}

\newcommand{\eE}{\mathbb{E}}
\newcommand{\PP}{\mathbb{P}}
\def\Nat{\mathbb{N}}

\def\liml{\lim\limits}

\def\intl{\int\limits}

\def\Dom{\text{\rm Dom}}

\def\ffi{\varphi}
\def\eps{\varepsilon}

\def\cR{{\mathbb R}}
\def\cRd{{\mathbb R}^d}
\def\cC{{\mathbb C}}

\def\pd{\partial}

\begin{document}

\title[Stochastic solutions of generalized time-fractional  equations]{Stochastic solutions of generalized time-fractional evolution equations}

\author{Christian Bender}
\address{Universit\"{a}t des Saarlandes, Fachrichtung Mathematik, Postfach 15 11 50, 66041 Saarbr\"{u}cken.} 
\email{\texttt{bender@math.uni-sb.de}}

\author{Yana A. Butko}
\address{ Technische Universit\"{a}t Braunschweig,  Institut f\"{u}r Mathematische Stochastik,  Universit\"{a}tsplatz 2, 38106 Braunschweig.}
\email{\texttt{y.kinderknecht@tu-bs.de}, \texttt{yanabutko@yandex.ru}}

\date{\today}

\maketitle

\begin{abstract}
We consider a general class of integro-differential evolution equations which includes the governing equation of the generalized grey Brownian motion and the time- and space-fractional heat equation. We present a general relation between the parameters of the equation and the distribution of the underlying stochastic processes, as well as  discuss different classes of  processes providing stochastic solutions of these equations. For a subclass of evolution equations, containing Saigo-Maeda generalized time-fractional operators, we  determine the parameters of the corresponding  processes explicitly. Moreover, we explain how self-similar stochastic solutions with stationary increments can be obtained via linear fractional 
L\'evy motion for suitable pseudo-differential operators in space.

\bigskip

\end{abstract}

%\tableofcontents

\section{Introduction}

%{\Large\noindent\textbf{State of the art}:} 
Einstein's explanation of Brownian motion has provided the cornerstone which underlies deep connections between stochastic processes and evolution equations. Namely, the function  $\quad  p_t(x):=(2\pi t)^{-d/2}\exp\left(   -\frac{|x|^2}{2t}\right),\quad$ which is the probability density function (PDF) of a ($d$-dimensional) Brownian motion $(B_t)_{t\geq0}$,  is also the fundamental solution of the heat equation  $\quad\frac{\partial u}{\partial t}(t,x)=\frac12\Delta u(t,x)\quad$ with the Laplace operator $\Delta$. In other words, the heat equation is the governing equation for Brownian motion. And the solution of the Cauchy problem for the heat equation with initial data $u_0$ has the stochastic representation  
\begin{equation*}
u(t,x)=\eE[u_0(x+B_t)].
\end{equation*}
 It\^{o} calculus for Brownian motion allows to extend the above relation to a wider class of evolution equations. Using It\^{o} calculus (together with martingale theory or with theory of Markov processes and operator semigroups), one can prove (under suitable assumptions), e.g., the  Feynman--Kac formula \cite{MR2848339,MR544188}
\begin{equation*}
u(t,x):=\mathbb{E}\left[u_0(x+B_t)\,e^{\,\int_0^t b(x+B_s)\cdot d B_s-\frac12\int_0^t|b(x+B_s)|^2ds +\int_0^t c(x+B_s)ds}\right],
\end{equation*}
for the  \emph{standard diffusion equation} with drift $b$ and a ``potential''/killing term $c$
\begin{equation}\label{example-eq}
\frac{\partial u}{\partial t}(t,x)=\frac12\Delta u(t,x)+b(x)\cdot \nabla u(t,x)+ c(x) u(t,x)
\end{equation}
 as well as  (quite analogous) Feynman--Kac formulae for %the
  Schr\"{o}dinger type counterparts of equation~\eqref{example-eq}, see e.g. \cite{MR1703898,ND,MR574173}. 
 
\emph{ Standard} or \emph{Brownian diffusion} is  identified by the linear growth in time of the variance and by a Gaussian shape of the displacement distribution. However, many natural phenomena show a diffusive behaviour that exibit a non-linear growth in time  of the variance and/or non-Gaussian shape of the displacement distribution; such phenomena are generally labeled as \emph{anomalous} (or, \emph{fractional})  \emph{ diffusion}. Anomalous diffusion is ubiquitously observed in many complex systems, ranging from turbulence and plasma physics to soft matter  (e.g., cell cytoplasm, membrane and nucleus) and neuro-physiological systems (see, e.g, \cite{MR1809268,MR3916448} and references therein). 
%The mathematical study of anomalous diffusion is nowadays a very actively developing branch in Analysis, Stochastics and Numerics with a huge amount of the literature on the subject.
 There are many  different mathematical models describing  anomalous diffusion. And, usually, these models lead to evolution equations generalizing equation~\eqref{example-eq} by substituting partial derivatives with respect to space and/or time by some non-local integro-differential operators (in particular, operators corresponding to fractional derivatives).   

One of the earliest and most well-studied models of anomalous diffusion is based on the \emph{Continuous Time Random Walk (CTRW)} approach (see, e.g.,~\cite{MR2551283,MR884309,MR2766141,MR2884383,MR172344,MR2097999}). 
The trajectory of each diffusing particle is considered to be governed by the  PDF $p_t(x)$ (of finding the particle in position $x$ at time $t$) which solves the Montroll--Weiss equation~\cite{MR172344}. Under some assumptions on jumps and waiting times of a CTRW, one obtains  in the proper scaling limit a symmetric $\gamma$-stable L\'{e}vy process time-changed (or, subordinated) by an independent  inverse $\beta$-stable subordinator. And the Montroll--Weiss equation leads to the time- and space-fractional heat equation 
\begin{align}\label{eq:time-space-fract-heat-eq}
u(t,x)=u_0(x)+\frac{1}{\Gamma(\beta)}\int_0^t (t-s)^{\beta-1}\left(\frac{1}{2}\Delta\right)^{\gamma/2}u(s,x)ds,\quad \beta\in(0,1],\quad\gamma\in(0,2],
\end{align}
as the governing equation for this process. Here $\left(\frac{1}{2}\Delta\right)^{\gamma/2}$ is the fractional Laplacian (up to a constant), i.e. a pseudo-differential operator with symbol $-2^{-\gamma/2}|p|^{\gamma}$. And the integral operator, which is applied to $  \left(\frac{1}{2}\Delta\right)^{\gamma/2} u(\cdot,x)$,  is the Riemann-Liouville fractional integral of order $\beta$. The  equation~\eqref{eq:time-space-fract-heat-eq}  can be rewritten also in the formalism of Caputo fractional derivatives in the following way
\begin{equation*}
\pd^\beta_t u(t,x)=\left(\frac{1}{2}\Delta\right)^{\gamma/2} u(t,x),
\end{equation*}
where $\pd^\beta_t f(t):= \int_0^t\frac{(t-s)^{-\beta}}{\Gamma(1-\beta)}\frac{d f(s)}{d s}ds$ is the Caputo fractional derivative of order $\beta$.  Many authors have contributed to various generalizations of the above results, showing that Markov processes, subordinated   by independent inverse subordinators, provide stochastic (representations of) solutions of some suitable time- (and, possibly, space-) fractional evolution equations (see, e.g.,~\cite{MR1874479,MR2766141,MR3987876} and references therein).

A new direction in theoretical modelling of diffusion in complex media interpretes the anomalous character of the diffusion as a consequence  of a very heterogeneous enviroment \cite{MR3280006,PhysRevLett.113.098302,MR3903618,Jain2017,MR3916448,Sposini_2018}. %Note that random diffusivity parameter, considered in some of such models, can be interpreted also as stochastic  volatility in 
%These models include Heterogeneous Diffusivity Models (HDMs) (\cite{MR3280006}); in particular, Diffusing Diffusivity Models (\cite{PhysRevLett.113.098302,Jain2017,Sposini_2018})and Superstatistics (\cite{MR2070694}).
 % According to this approach, a classical diffusion takes plase for a single particle, whereas the fractionality emerges from the population of scales characterizing the medium. 
   %In particular, Diffusing Diffusivity Models consider  random diffusivity parameters \cite{PhysRevLett.113.098302,Jain2017,Sposini_2018}.
   One type of such models is based on \emph{randomly scaled Gaussian processes} (RSGP)  and is  sometimes refered to as \emph{Generalized Grey Brownian Motion} (GGBM), or \emph{GGBM-like models}.  This type of models originates from Schneider's grey Brownian motion \cite{MR1124240,MR1190506} whose PDF is  the fundamental solution of the time-fractional heat equation of order $\beta$ (i.e., equation~\eqref{eq:time-space-fract-heat-eq} with $\gamma:=2$). The GGBM $(X^{\alpha,\beta}_t)_{t\geq0}$, $\alpha\in(0,2)$, $\beta\in(0,1]$, was introduced in works of Mainardi, Mura and their coauthors  \cite{MR2501791,MR2430462,MR2588003}, and can be realized as 
\begin{align}\label{eq:GGBM}
X^{\alpha,\beta}_t:=\sqrt{A_\beta}B^{\alpha/2}_t,
\end{align}   
 where $B^{\alpha/2}_t$ is a ($1$-dim.) fractional Brownian motion (FBM) with Hurst parameter $\alpha/2$ and $A_\beta$ is a particular nonnegative random variable\footnote{ The distribution of the random variable  $A_\beta$ has    Laplace transform $E_\beta(-\cdot)$, where    $E_\beta(z):=\sum_{n=0}^\infty\frac{z^n}{\Gamma(\beta n+1)}$ is the the Mittag--Leffler function. For $\beta\in(0,1)$, the PDF of $A_\beta$ is given by the Mainardi-Wright function   $\mathcal{M}_\beta(x):=\sum_{n=0}^\infty \frac{(-x)^n}{n!\Gamma(1-\beta-\beta n)}$. } which is independent from  $(B^{\alpha/2}_t)_{t\geq0}$. The GGBM includes Brownian motion (for $\alpha=\beta=1$), FBM (for $\alpha\neq\beta=1$) and Schneider's grey Brownian motion (for $\alpha=\beta\neq1$).  These processes are self-similar and  with stationary increments, what makes the GGBM  attractive for modeling. Generally, GGBM-like models deal with processes of the form $\sqrt{A} G_t$, or $\sqrt{A_t}G_t$, where $G_t$ is a Gaussian process and $A$ (or $A_t$) is an independent nonnegative random variable (or process) \cite{MR3586912}. The PDF of GGBM is the fundamental solution %(i.e. the solution with $u_0:=$ Dirac delta-function)
    for the following evolution equation (cf., e.g., \cite{MR3464056,MR2501791}): % (using the formalism of Caputo fractional derivatives, cf.  \cite{MR3400947}): 
\begin{equation}\label{GGBM-eq}
u(t,x)=u_0(x)+\frac{\alpha}{\beta\Gamma(\beta)}\int_0^t s^{\frac{\alpha}{\beta}-1}\left(t^{\frac{\alpha}{\beta}}-s^{\frac{\alpha}{\beta}}  \right)^{\beta-1}\frac12\frac{\pd^2 u(s,x)}{\pd x^2}ds,
\end{equation}
which is the  time-stretched time-fractional heat equation.
For $\alpha:=\beta$, equation~\eqref{GGBM-eq} reduces to the time-fractional heat equation of order $\beta$. 
%
%The governing equation~\eqref{GGBM-eq} for the PDF of GGBM can be rewritten also in the formalism of Caputo fractional derivatives in the following way (cf.  \cite{MR3400947}): 
%\begin{equation}\label{GGBM-eq_Caputo}
%\int_0^t\frac{\pd u(s,x)}{\pd s}\cdot\frac{(t^{\alpha/\beta}-s^{\alpha/\beta})^{-\beta}}{\Gamma(1-\beta)}ds=\frac12\frac{\pd^2 u(t,x)}{\pd x^2},
%\end{equation}
%so that, for $\alpha:=\beta$, we have the Caputo derivative $\pd^\beta_t$ of order $\beta$ on the left hand side of~\eqref{GGBM-eq_Caputo}.  
 Mathematical theory of GGBM is actively developing nowadays \cite{doi:10.1080/17442508.2019.1641093,MR3733422,MR3854539,MR3956716,MR3400947,MR3464056}. 
%However,   no  Feynman--Kac formulae of type~\eqref{example-FFK} on the base of GGBM-like processes are  presented in the literature by the moment.
Moreover, another GGBM-like process of the form $\sqrt{A} G_t$ is shown to have PDF which solves the (1-dimensional) time- and space-fractional heat equation~\eqref{eq:time-space-fract-heat-eq} (see~ \cite{MR3513003}).
%$\quad\pd^\beta_t u(t,x)=$ $\Delta^{\gamma}u(t,x)\,\,$ with Caputo fractional time derivative $\pd^\beta_t$ and ($1$-dim.) fractional Laplacian $\quad \Delta^{\gamma}$, $\gamma\in(0,1]$, cf. \cite{MR3513003}.  

The fact that the time- and space-fractional heat equation~\eqref{eq:time-space-fract-heat-eq} serves as the governing equation for so different stochastic processes as a RSGP   and a stable L\'{e}vy process subordinated by an inverse stable subordinator motivetes the following questions: What is in common between these two classes of processes? What other classes of processes  can be used to solve such type of evolution equations? How general can be evolution equations allowing such types of stochastic solutions? 

The present paper  gives our answers to these questions. We consider a general class of evolution equations of the form
\begin{align}\label{eq:EvEq}
u(t,x)&=u_0(x)+\int_0^t k(t,s)Lu(s,x)ds,\qquad t>0,\quad x\in\cRd, \\
\lim_{t\searrow 0} u(t,x)&=u_0(x),\qquad  x\in\cRd,\nonumber
\end{align}
where $L$ is a pseudo-differential operator associated to a L\'evy process and 
$k(t,s)$, $0<s<t<\infty$, is a general kernel. This setting
largely extends equations~\eqref{eq:time-space-fract-heat-eq} and~\eqref{GGBM-eq}. In Section~\ref{sec:our approach}, we present a general relation between the parameters of the equation \eqref{eq:EvEq} and the distribution of any stochastic process, 
which provides a stochastic solution of Feynman-Kac type. The proof of this relation is presented in Section~\ref{sec:proof}. More precisely, we derive 
a series representation in terms of the time kernel $k$ and the symbol $-\psi$ of the pseudodifferential operator $L$ for the characteristic function of the one-dimensional marginals of any stochastic solution. We explain how this series simplifies in the important case of homogeneous kernels 
which includes the kernel $k(t,s)=(t-s)^{\beta-1}/\Gamma(\beta)$ for time-fractional evolution equations and, more generally, kernels 
corresponding to  Saigo-Maeda fractional diffintegration operators. The connection between 
Saigo-Maeda fractional diffintegration operators and positive random variables with Laplace transform given by Prabhakar's three parameter generalization of the Mittag-Leffler function is established 
in Section ~\ref{sec:special case}. The results of Section~\ref{sec:special case} yield a stochastic representation for \eqref{eq:EvEq}
with a Saigo-Maeda kernel in terms of a randomly slowed-down / speeded-up L\'evy process $(Y_{At^\beta})_{t\geq 0}$, where $Y$ is a L\'evy process with infinitesimal generator $L$, $A$ is an independent random variable
with Laplace transform given by the three-parameter Mittag-Leffler function, and $\beta$ corresponds to the degree of homogeneity of the kernel. If $Y$ has a stable distribution (e.g., in the case when $L$ is a symmetric fractional Laplacian in space), the randomly slowed-down / speeded-up L\'evy process can be replaced 
by a randomly scaled linear fractional stable motion,  providing a stochastic solution in terms of a self-similar process with stationary increments, as demonstrated in Section \ref{sec:fsm}.

%%%%%%%%%%%%%%%%%%%%%%%%%%%%%%%%%%%%%%%%%%%%%%%          Section 1         %%%%%%%%%%%%%%%%%%%%%%%%%%%%%%%%%%%%%%%%%%%
\section{Integro-differential evolution equations and  stochastic processes providing stochastic representations of their solutions}\label{sec:our approach}

In this section, we state our main results on the existence of stochastic representations for generalized 
evolution equations of the form \eqref{eq:EvEq},
where $L$ is a pseudo-differential operator satisfying Assumption \ref{ass:u_0+psi} below. 
We first consider general time kernel functions $k(t,s)$, $0<s<t<\infty,$ and then specialize to the cases 
of homogeneous kernels and convolution kernels, respectively.

\subsection{General time kernels}
We first fix some notation.
Let $$S(\cRd):=\left\{\ffi\in C^\infty(\cRd)\,:\,\liml_{|x|\to\infty} x^\alpha \partial^\beta\ffi(x)=0\,\,\forall\alpha,\beta\in\Nat_0^d  \right\}$$ be the Schwartz space. We consider a continuous negative definite function\footnote{Each CNDF $\psi\,:\cRd\to\cC$ is uniquely determined by its L\'{e}vy-Khintchine representation
$$
\psi(p)=c-ip\cdot b+\frac12 p\cdot Qp+\int_{\cRd\setminus\{0\}}\left( 1-e^{ip\cdot y}+\frac{ip\cdot y}{1+|y|^2} \right)\nu(dy),
$$
where $c\geq0$, $b\in\cRd$, $Q$ is a symmetric positive semidefinite $d\times d$-matrix, $\nu$ is a measure on $\cRd\setminus\{0\}$ such that $\int_{\cRd\setminus\{0\}}\min(|y|^2,1)\nu(dy)<\infty$.  Any such quadruple $(c,b,Q,\nu)$ defines a CNDF. It follows that each CNDF $\psi$ satisfies the estimate $|\psi(p)|\leq C_\psi(1+|p|^2)$ for all $p\in\cRd$ and some $C_\psi\geq0$. 
} 
(CNDF) $\psi\,:\cRd\to\cC$ and the pseudo-differential operator $(L,S(\cRd))$  with symbol $-\psi$ in the Banach space $C_\infty(\cRd)$ of continuous functions vanishing at infinity with supremum-norm $\|\ffi\|_\infty:=\sup_{x\in\cRd}|\ffi(x)|$ (cf. Example~4.1.16 of~\cite{MR1873235}), i.e. for each $\ffi\in S(\cRd)$
\begin{align}\label{eq:L}
L\ffi(x):=\left(\mathfrak{F}^{-1}\circ (-\psi)\circ\mathfrak{F}\ffi\right)(x)\equiv-(2\pi)^{-d}\int_{\cRd}\int_{\cRd}e^{ip\cdot(x-q)} \psi(p)\ffi(q)dqdp, 
\end{align}
where $\mathfrak{F}$ is the Fourier transform such that $\mathfrak{F}\ffi(p)=(2\pi)^{-d/2}\int_{\cRd} e^{-ipx}\ffi(x)dx$. Note that the operator $(L,S(\cRd))$ is closable in the  space  $C_\infty(\cRd)$ and the closure $(L,\Dom(L))$, $\Dom(L):=\left\{\ffi\in C_\infty(\cRd)\,:\,L\ffi\in C_\infty(\cRd)\right\}$,  generates a strongly continuous semigroup on $C_\infty(\cRd)$. This semigroup (or the operator $(L,\Dom(L))$ itself) corresponds to a L\'{e}vy process $Y:=(Y_t)_{t\geq0}$ with charachteristic exponent $\psi$, i.e. $\eE\left[\exp(ip\cdot Y_t)  \right]=\exp(-t\psi(p))$ (see, e.g. \cite{MR2512800,MR1873235}).

\begin{assumption}\label{ass:u_0+psi}
We assume that the initial data  $u_0$ is an arbitrary function from $S(\cRd)$ and that the symbol $-\psi$ of the operator $L$ satisfies $\psi(0)=0$, i.e. there is no killing term (i.e. $c=0$) in the L\'{e}vy-Khintchine representation of $\psi$ (and hence an underlying L\'{e}vy process $(Y_t)_{t\geq0}$ has an infinite life time).
\end{assumption}

Concerning general kernels, we impose the following condition:

%\ykm{$\alpha$ in Ass~2.2 ist durch $\alpha^*$ \"uberall ersetzt}%%%%%%%%%%%%%%%%%%%%%%%%%%%%%%%%%%%%%%%%%%%%%%%%%%%%%%
\begin{assumption}\label{ass:r}
 We consider a Borel-measurable kernel $k\,:\,(0,\infty)\times(0,\infty)\to\cR$ satisfying the following condition:  $\exists\,\alpha^*\in[0,1)$ and $\exists\,\eps>0$ such that for each $T>0$ 
\begin{align*}
K_T:=\sup\limits_{0<t\leq T}t^{\alpha^*-\frac{1}{1+\eps}}\|k(t,\cdot)\|_{L^{1+\eps}((0,t))}<\infty.
\end{align*}
%\item[(ii)] \ykm{I don't see, where we need it} \yk{the mapping $t\mapsto k(t,ts)$ is continuous on $(0,\infty)$ for (almost) all $s\in(0,1)$.} 
\end{assumption}

%We consider the following integro-differential evolution equation (compare with~\eqref{GGBM-eq}):
%\begin{align}\label{eq:EvEq}
%u(t,x)=u_0(x)+\int_0^t k(t,s)Lu(s,x)ds,\qquad t>0,\quad x\in\cRd.
%\end{align}
%\ykm{Solution is classical w.r.t. $t$ and $L^2$ w.r.t. $x$?!?!}
%with initial data $u_0\in  S(\cR)$,
%% where, for each  fixed $t>0$, the equality of the left and the right hand sides of  equation~\eqref{eq:EvEq} is understood in the norm of the space $L^2(\cRd)$.  

Continuity of stochastic representations for \eqref{eq:EvEq} can be obtained under the following continuity assumption on the kernel: 

\begin{assumption}\label{ass:cont}
 For a.e. $s\in(0,1)$ the mapping
$$
(0,\infty)\rightarrow \mathbb{R},\quad t\mapsto k(t,ts)
$$
is continuous.
\end{assumption}

%\yk{Here we need to discuss the uniqueness of the solution $u(t,x)$ of the equation~\eqref{eq:EvEq} (should follow from the theory of Volterra equations), in order to emphasize that all our different stochastic representations are representations of the same solution $u$.}

Let $(\Omega,\mathcal{F},\mathbb{P})$ be a probability space and $X:=(X_t)_{t\geq0}$ be an $\cRd$-valued stochastic process on it such that  $X_0=0$ $\,\,\PP$-almost surely.  We shall say that $X$ \emph{provides a stochastic solution} to \eqref{eq:EvEq}, 
if the function $u$ defined by
\begin{align}\label{eq:u(t,x)}
u(t,x)=\mathbb{E}\left[u_0\left( x+X_t\right)\right],\qquad x\in\cRd,\quad t\geq0,
\end{align}
is a solution to \eqref{eq:EvEq} for every $u_0\in S(\cRd)$. 

The following theorem characterizes, when such a stochastic solution exists.

\begin{theorem}\label{thm:generalTheorem}
Let Assumptions~\ref{ass:u_0+psi}, \ref{ass:r} hold. 
  For each $t\geq0$, the  function $\Phi(t,\cdot)\,:\,\cC\to\cC$ given by
\begin{align}
&\Phi(t,\lambda):=\sum_{n=0}^\infty c_n(t)\lambda^n, \label{eq:Phi}\\
&
c_0(t):=1\qquad\forall\,\,t\geq0\nonumber\qquad\text{ and}\\
&
 c_n(t):=\left\{\begin{array}{ll}
\int_0^t k(t,s)c_{n-1}(s)ds, & \quad\forall\,\,t>0,\\
0 & \quad t=0,
\end{array}\right.
\qquad n\in\Nat,\label{eq:coeff of Phi}
\end{align}
is well-defined (i.e., the integrals in the recursion formula exist) and entire.
Moreover:

\noindent (i) A stochastic solution to \eqref{eq:EvEq} exists, if and only if the function $\Phi(t,-\psi(\cdot))$ is  positive definite\footnote{By the Bochner theorem, a function $f\,:\,\cRd\to\cC$ is the Fourier transform of a bounded measure on $\cRd$ if and only if $f$ is continuous and positive definite (see, e.g.,~\cite{MR1873235}).} for all $t\geq0$.
In this case,  $(X_t)_{t\geq0}$ provides a stochastic solution to  \eqref{eq:EvEq}, if and only if 
$$
\eE\left[ e^{ip\cdot X_t} \right]=\Phi(t,-\psi(p)),\qquad p\in\cRd,\quad t\geq0.
$$

\bigskip

\noindent (ii) If the restriction of the function $\Phi(t,-\cdot)$ on $(0,\infty)$ is completely monotone\footnote{A  function $f : (0, \infty) \to [0, \infty)$ is said to be a completely monotone function  if $(-1)^n f^{(n)}(x) \geq 0$ for $x > 0$ and $n\in\Nat_0$. A function $f$ is completely monotone if and only if it is the Laplace transform of a  measure on the half-line $[0,\infty)$ (Bernstein's theorem).} for all $t\geq0$, the process $(Y_{A(t)})_{t\geq0}$ provides a stochastic solution to \eqref{eq:EvEq}, where, for each $t\geq0$, $A(t)$ is a non-negative random variable whose distribution $\mathcal{P}_{A(t)}$ has the Laplace transform given by $\Phi(t,-\cdot)$, i.e. $\int_0^\infty e^{-\lambda a}\mathcal{P}_{A(t)}(da)=\Phi(t,-\lambda)$, and $(Y_t)_{t\geq0}$ is a L\'{e}vy process with characteristic exponent $\psi$ which is independent from $(A(t))_{t\geq0}$.

\bigskip

\noindent (iii) If, additionally, the characteristic exponent $\psi$ is given by $\psi:=f\circ\widetilde{\psi}$ for some other CNDF $\widetilde{\psi}$ and some Bernstein function\footnote{A continuous function $f : (0, \infty) \to [0, \infty)$ is said to be a Bernstein function (BF) if $(-1)^n f^{(n)}(x) \leq 0$ for $x > 0$ and $n\in\Nat$. Note that a composition BF$\circ$CNDF is again a CNDF, applying the extension of Bernstein functions to the complex numbers with nonnegative real part 
as decsribed in \cite{MR2978140}, Proposition 3.5.
%Clearly, $f$ is a Bernstein function if and only if it is nonnegative, and $f'$ is a completely monotone function.
%Every Bernstein function has an extension $f\,:\,\overline{\Pi}_+\to\overline{\Pi}_+$, $\Pi_+:=\left\{z\in\cC\,:\,\Re z>0    \right\}$ and $\overline{\Pi}_+:=$ closure of $\Pi_+$, which is continuous on $\overline{\Pi}_+$ and holomorphic on $\Pi_+$.
 } (BF) $f$, then the 
 process $(\tilde Y_{\tilde A(t)})_{t\geq0}$ provides a stochastic solution to \eqref{eq:EvEq}, where,
 for each $t\geq0$, $\widetilde{A}(t)$ is a non-negative random variable whose distribution $\mathcal{P}_{\widetilde{A}(t)}$ has the Laplace transform given by the function $\Phi(t,-f(\cdot))$, and $(\widetilde{Y}_t)_{t\geq0}$ is a L\'{e}vy process with  charachteristic exponent $\widetilde{\psi}$ which is independent from $(\widetilde{A}(t))_{t\geq0}$. 
 \end{theorem}

The proof of Theorem~\ref{thm:generalTheorem} as well as of the other results of this section will be provided in Section \ref{sec:proof}.

The following remark briefly explains how to deal with time-stretched equations.
\begin{remark}\label{rem:decoration}
Consider the following class of time-stretchings (or ``dressings'', cf. Sec.~3.3 of~\cite{MR3987876}) 
$
\mathcal{G}:=\big\{g\,:[0,\infty)\to[0,\infty)$ such that $g(\tau)>0  \,\,\forall\,\tau>0$,    $g(\tau)\nearrow\infty$ as $\tau\nearrow\infty$, $g(\tau)=\int_0^\tau \dot{g}(\theta)d\theta$  for some  $\dot{g}\in L^1_{loc}([0,\infty))$ with $\dot{g}(\tau)>0\,\,\forall\,\tau>0    \big\}$. Using the change of variables $t=g(\tau)$, we obtain the analogue of Theorem~\ref{thm:generalTheorem} for the whole class of time-stretched   equations  %(or equations with ``dressed time-fractional operators'')
\begin{align}\label{eq:decoratedEvEq}
v(\tau,x)=u_0(x)+\int_0^\tau\kappa_g(\tau,\theta) Lv(\theta,x)d\theta,\qquad\tau>0,\quad x\in\cRd,\quad g\in\mathcal{G},
\end{align}
where the kernel $\kappa_g$ is defined via
\begin{align}\label{eq: k decorated}
\kappa_g(\tau,\theta):=k(g(\tau),g(\theta))\dot{g}(\theta).
\end{align}
Obviously, a function $v$ solves evolution equation~\eqref{eq:decoratedEvEq} if and only if $v(\tau,x)=u(g(\tau),x)$, where $u$ solves the evolution equation~\eqref{eq:EvEq} with the corresponding kernel $k$. And $v(\tau,x)=\eE\left[u_0\left(x+X_{g(\tau)}\right)\right]$ solves \eqref{eq:decoratedEvEq}, if $(X_t)_{t\geq0}$ provides a stochastic solution to~\eqref{eq:EvEq}.  Note, that the kernel in the GGBM-equation~\eqref{GGBM-eq} can be obtained from the time-fractional kernel of equation~\eqref{eq:time-space-fract-heat-eq} via the time-stretching $g(\tau):=\tau^{\alpha/\beta}$.
\end{remark}

We close this subsection by establishing a sufficient condition for continuity of $\Phi$ and $u$ in both variables. 
\begin{proposition}\label{prop:cont}
 Suppose Assumptions \ref{ass:r} and \ref{ass:cont} are in force. Then $t\mapsto c_n(t)$ is continuous on $[0,\infty)$ for every $n\in \mathbb{N}$ and $\Phi$ is continuous on
$[0,\infty)\times \mathbb{C}$. If, moreover, Assumption \ref{ass:u_0+psi} holds 
and if a process $(X_t)_{t\geq 0}$ provides a stochastic solution to \eqref{eq:EvEq}, then the solution $u$ given by \eqref{eq:u(t,x)} is continuous.
\end{proposition}

\subsection{Homogeneous kernels}

We now explain how the general results simplify in the case of a homogeneous kernel. Recall that a kernel function $k$ is \emph{homogeneous of degree} $\beta-1$, if 
$$
k(t,ts)=t^{\beta-1}k(1,s)
$$
for every $t\in (0,\infty)$ and $s\in(0,1)$

%\ykm{We do not need the restriction $\beta\leq1$ in Thm.~2.2!!!}%%%%%%%%%%%%%%%%%%%%%%%%%%%%%%%%%%%%%%%%%%%%%%%%%%%%%%%%
\begin{theorem}\label{thm:hom}
 Suppose $k$ is homogeneous of degree $\beta-1$ for some $\beta>0$ and $k(1,\cdot)\in L^{1+\varepsilon}((0,1))$ for some $\varepsilon>0$. Then, Assumptions \ref{ass:r} and \ref{ass:cont} are satisfied and $\Phi$ takes the form 
 $$
 \Phi(t,\lambda)=\hat \Phi(\lambda t^\beta), \qquad t\geq 0, \quad \lambda\in \mathbb{C}, 
 $$
 where $\hat \Phi(\lambda)=\sum_{n=0}^\infty \hat c_n \lambda^n$ and 
 $$
  \hat c_n:=
\hat c_{n-1} \int_0^1 k(1,s)s^{\beta(n-1)}ds, 
\qquad n\in\Nat,\qquad \hat c_0=1.
 $$
 Additionally, suppose that Assumption \ref{ass:u_0+psi} holds and $\psi=(\tilde \psi)^\gamma$ 
 for some $\gamma\in (0,1]$ and a CNDF $\tilde \psi$, and denote by $\tilde Y$ a L\'evy process with charachteristic exponent $\tilde \psi$. 
 If $x\mapsto \hat \Phi(-x)$ is completely
 monotone on $(0,\infty)$, then there is a nonnegative random variable $\tilde A$ independent of $\tilde Y$ with Laplace transform 
 $\hat \Phi(-(\cdot)^\gamma)$  and $( \tilde Y_{\tilde A t^{\beta/\gamma}})_{t\geq 0}$ 
 provides a stochastic solution to \eqref{eq:EvEq}. 
\end{theorem}

We again postpone the proof to Section \ref{sec:proof}, but explain how to obtain conditionally Gaussian representations for the fractional heat 
equation in time and space from this result.
\begin{example}\label{ex:fractional}
 The fractional time kernel
 $$
 k(t,s)=\frac{1}{\Gamma(\beta)} (t-s)^{\beta-1},\qquad\beta\in(0,1],
 $$
 where $\Gamma$ denotes the gamma function,
 is homogeneous of degree $\beta-1$. Since
 $$
 \int_0^1 k(1,s)s^{\beta(n-1)}ds=\frac{\Gamma((n-1)\beta+1)}{\Gamma(n\beta+1)},
 $$
 we obtain 
 $$
 \hat c_n=\frac{1}{\Gamma(n\beta+1)}
 $$
 and, thus,
 $$
 \hat \Phi(\lambda)=E_{\beta}(\lambda):=\sum_{n=0}^\infty \frac{\lambda^n}{\Gamma(n\beta+1)}
 $$
 is the Mittag-Leffler function. Hence   $\Phi(t,\lambda)=E_\beta(t^\beta\lambda)$.  In a similar way, the GGBM-kernel
 \begin{align*}
 k(t,s):=\frac{\alpha}{\beta\Gamma(\beta)}s^{\frac{\alpha}{\beta}-1}\left(t^{\frac{\alpha}{\beta}} -s^{\frac{\alpha}{\beta}} \right)^{\beta-1},\qquad\beta\in(0,1],\,\,\alpha\in(0,2),
 \end{align*}
 is homogeneous of degree $\alpha-1$ and the corresponding $ \Phi(t,\lambda)=E_{\beta}(t^{\alpha}\lambda)$  (cf.  Remark \ref{rem:decoration}). The function   $E_{\beta}(-\cdot)$ is known to be completely monotone for $\beta\in (0,1]$ since the work by Pollard \cite{MR0027375}, and we denote by $A_\beta$ a nonnegative random variable which has $E_{\beta}(-\cdot)$ as 
 Laplace tansform. Suppose $B$ is a $d$-dimensional standard Brownian motion independent of $A_\beta$. %\ykm{new notation $A_\beta$}
  By
 Theorem \ref{thm:hom} with $\gamma=1$, the solution to the time-stretched time-fractional heat equation
 \begin{align}\label{eq:againGGBM-eq}
 u(t,x)=u_0(x)+\frac{\alpha}{\beta\Gamma(\beta)}\int_{0}^t s^{\frac{\alpha}{\beta}-1}\left(t^{\frac{\alpha}{\beta}}-s^{\frac{\alpha}{\beta}}  \right)^{\beta-1}\frac12 \Delta u(s,x) ds
 \end{align}
(which is the governing equation for the GGBM)  has the stochastic representation 
\begin{align}\label{eq:stochRepr}
 \mathbb{E}[u_0(x+B_{A_\beta t^\alpha})].
 \end{align}
   Similarly, for $\beta\in(0,1]$ and $\gamma\in(0,1)$ denote by $A_\beta^{(\gamma)}$ a random variable with Laplace transform $E_\beta(-(\cdot)^\gamma)$ independent of $B$. Then, by Theorem \ref{thm:hom},
 \begin{align}\label{eq:solutionTime-spaceFracHeatEq}
 \mathbb{E}\left[u_0\big(x+B_{A_\beta^{(\gamma)} t^{\beta/\gamma}}\big)\right]
 \end{align}
 provides a stochastic representation for the time-space fractional heat equation (with symbol $\psi(p)=|p|^{2\gamma}/2^{\gamma}$)
 \begin{align}\label{eq:time-spaceFracHeatEq}
 u(t,x)=u_0(x)+\frac{1}{\Gamma(\beta)}\int_{0}^t \left(t-s  \right)^{\beta-1} \left(\frac12 \Delta\right)^\gamma u(s,x) ds.
 \end{align}
 As the distribution of a Gaussian random vector is determined by its mean and covariance function, we may replace 
 $B_{A_\beta t^\alpha}$ 
 in the stochastic representation~\eqref{eq:stochRepr} for the solution of the GGBM-equation~\eqref{eq:againGGBM-eq} with $\alpha\in(0,2)$ by a multivariate extension of generalized grey Brownian motion
 $$
 X^{\alpha,\beta}_t:=\sqrt{A_\beta } B^{\alpha/2}_t
 $$
 where $B^H$ is a $d$-dimensional fractional Brownian motion with Hurst parameter $H\in (0,1]$ and independent of $A_\beta$, thus obtaining the time-stretched time-fractional heat equation as the governing equation for generalized grey Brownian motion also in the multivariate case. We recall here that, for $d=1$, a 1-dimensional fractional Brownian motion is a centred Gaussian process with covariance structure
 $$
 \eE\left[B^H_t B^H_s\right]=\frac{1}{2}\left(t^{2H}+s^{2H}-|t-s|^{2H}\right),\quad t,s\in [0,\infty),
 $$
 and that for $d>1$ the components of $B^H$ are independent 1-dimensional fractional Brownian motions with Hurst parameter $H$.
 
 While this representation is appealing 
 from a modeling point of view, because fractional Brownian motion (and, thus, generalized grey Brownian motion) is self-similar with stationary increments, simple representations in terms of a Brownian motion such as $B_{A_\beta t^\alpha}$ or $\sqrt{A_\beta} B_{t^\alpha}$ are mathematically convenient as they allow to adopt tools from martingale theory and from the theory of Markov processes.
 
 Analogously, for representing the solution to the time-space fractional heat equation~\eqref{eq:time-spaceFracHeatEq}, we may replace  
 $B_{A_\beta^{(\gamma)} t^{\beta/\gamma}}$ in~\eqref{eq:solutionTime-spaceFracHeatEq} by $\sqrt{A_\beta^{(\gamma)}} B^{\beta/(2\gamma)}_t$ (for $\beta\leq 2\gamma$), extending the stochastic representation with stationary increments of \cite{MR3513003} beyond the univariate case, or by 
 $\sqrt{A_\beta^{(\gamma)}} B_{t^{\beta/\gamma}}$ (without any additional restrictions on the relation between $\beta$ and $\gamma$).
\end{example}

\begin{remark}
The results of the previous example will be generalized in various directions in Sections \ref{sec:special case} and \ref{sec:fsm} below. In Section \ref{sec:special case}, we consider kernels corresponding 
to Saigo-Maeda fractional diffintegration operators and demonstrate how $\Phi$ relates to Prabhakar's three parameter generalization of the Mittag-Leffler function in this case. In Section \ref{sec:fsm} we discuss 
how to obtain stochastic representations with stationary increments beyond the Gaussian case in terms of fractional stable motion.
\end{remark}

\subsection{Convolution kernels}

For convolution kernels, we can derive the following result from Theorem \ref{thm:generalTheorem}.

\begin{theorem}\label{thm:conv}
 Suppose $k(t,s)=\mathfrak{K}(t-s)$, where $\mathfrak{K}:(0,\infty)\rightarrow \mathbb{R}$ is continuous and satisfies 
 $$
 |\mathfrak{K}(t)|\leq M t^{\beta-1}e^{\gamma t},\quad t>0,
 $$
 for some constants $M,\gamma\geq 0$ and $\beta \in (0,1]$. Let $(\mathcal{L}\mathfrak{K})(\cdot)$ be the Laplace transform of $\mathfrak{K}$.  If $(A(t))_{t\geq 0}$ is a nonnegative stochastic process with a.s. RCLL paths such that  
 $$
 \int_0^\infty e^{-\sigma t} \mathbb{E}\left[e^{-\lambda A(t)}\right]dt= \frac{1}{\sigma}\frac{1}{1+\lambda (\mathcal{L}\mathfrak{K})(\sigma)}
 $$
 for every $\lambda\geq 0$ and sufficiently large $\sigma\geq \sigma_0(\lambda)$, then $\Phi(t,-\cdot)$ is CM for every $t\geq 0$. If, moreover, Assumption \ref{ass:u_0+psi} is in force, then $(Y_{A(t)})_{t\geq 0}$ provides a stochastic solution to \eqref{eq:EvEq}, where $(Y_{t})_{t\geq 0}$ is a L\'evy process with 
 charachteristic exponent $\psi$ independent of $(A(t))_{t\geq 0}$.
\end{theorem}

\begin{example}\label{rem:inverse subordinators}
(i) Suppose $k(t,s)=\mathfrak{K}(t-s)$ is as in the previous theorem and the Laplace transform of $\mathfrak{K}$ equals $1/h$ for some BF $h$.
Then, $h$ is the Laplace exponent of some L\'{e}vy subordinator $(\eta^h_t)_{t\geq0}$. The corresponding inverse subordinator $(E^h_t)_{t\geq0}$ is defined via $E^h_t:=\inf\left\{ s>0\,:\,\eta^h_s>t\right\}$.   It has been shown in~\cite{MR2442372} (formula~(3.14)) that (in the case when the L\'{e}vy measure $\nu$ of $(\eta^h_t)_{t\geq0}$ satisfies $\nu(0,\infty)=\infty$) the double Laplace  transform %$\mathfrak{L}$ 
of the distribution $\mathcal{P}_{E^h_t}(da)$ with respect to both time and space variables is equal to
$$
\int_0^\infty e^{-\sigma t} \mathbb{E}\left[e^{-\lambda E^h(t)}\right]dt=\frac{h(\sigma)}{\sigma(h(\sigma)+\lambda)}= \frac{1}{\sigma}\frac{1}{1+\lambda (\mathcal{L}\mathfrak{K})(\sigma)}.
$$
Hence, the previous theorem recovers the well-known result that $\mathbb{E}[u_0(x+Y_{E^h(t)})]$ solves \eqref{eq:EvEq}, where $Y$ is a Levy process with characteristic exponent $\psi$, 
see e.g. \cite{MR2766141}.
\\[0.2cm] 
(ii) Let $\mathfrak{K}(s)=s^{\beta-1}/\Gamma(\beta)$ for some $\beta\in (0,1)$. We, hence, again consider the time-fractional evolution equation. Then,
$$
(\mathcal{L}\mathfrak{K})(\sigma)=\frac{1}{\sigma^\beta},\qquad \sigma>0,
$$
where $h(\sigma):=\sigma^\beta$ is the Laplace exponent of the $\beta$-stable subordinator, recovering the representation for the solution of time-fractional evolution equations in terms of a 
L\'evy process time-changed by an inverse stable subordinator, see e.g. \cite{MR1874479}. We have seen in Example~\ref{ex:fractional} 
that in this situation the time change can alternatively be done by $A(t)=A_\beta t^\beta$. This can also be verified by Theorem \ref{thm:conv}, because for $\sigma>0$ and $\lambda\geq 0$,
$$
\int_0^\infty e^{-\sigma t} \mathbb{E}[e^{-\lambda A_\beta t^\beta}]dt= \int_0^\infty e^{-\sigma t} E_\beta(-\lambda t^\beta)dt=\frac{\sigma^{\beta-1}}{\sigma^\beta+\lambda}=\frac{1}{\sigma}\frac{1}{1+\lambda (\mathcal{L}\mathfrak{K})(\sigma)},
$$
e.g. by Eq. (7.1) in \cite{MR2800586}.
\end{example}

\begin{remark}
Inverse subordinators are actively used to produce stochastic representations for solutions of evolution equations of the form~\eqref{eq:EvEq} with convolution kernels  also in the case when the generator of a L\'{e}vy process $(L,\Dom(L))$ is substituted by an arbitrary generator of a strongly continuous semigroup (see, e.g.~\cite{MR1874479,MR2766141,MR3987876} as well as works of  other authors). The statement~(iii) of Theorem~\ref{thm:generalTheorem}  also can be generalized to this case applying different tools. This topic will be presented in our next paper.
\end{remark}
%\yk{\begin{theorem}\label{thm:Markov subordination}
%Let $(M,\rho)$ be a metric space. Let $(\xi_t)_{t\geq0}$ be a Markov process with state space $(M,\rho)$ such that the family of operators $(T_t)_{t\geq0}$,
%$$
%T_tu_0(x):=\eE^x\left[ u_0(\xi_t)\right],\qquad u_0\in X,\quad x\in M,
%$$
%is a strongly continuous semigroup on some Banach space $X$ of functions of $x$-variable. Let $(L,\Dom(L))$ be the generator of $(T_t)_{t\geq0}$.  Let Assumption~\ref{ass:r} hold. Let $(A(t))_{t\geq0}$ be a family of non-negative random variables which are independent from $(\xi_t)_{t\geq0}$ and such that the Laplace transforms $\Phi(t,-\cdot)$ of their distributions $\mathcal{P}_{A(t)}$, $t\geq0$, are given by~\eqref{eq:Phi}-\eqref{eq:coeff of Phi}.
%Then the function $u$ given by
%\begin{align*}
%u(t,x):=\eE^x\left[u_0\left(\xi_{A(t)}\right)  \right]
%\end{align*}
%solves the evolution equation 
%\begin{align*}
%u(t,\cdot)=u_0(\cdot)+\int_0^t k(t,s) Lu(s,\cdot)ds,
%\end{align*}
%where the equality is the equality of elements of the Banach space $X$.
%\end{theorem}
%}

%\yk{The proof will be given in Section~\ref{sec:proof}.}

%The rest  of this Section is devoted to the proof of Theorem~\ref{thm:generalTheorem}.\ykm{Or as  the next Section...}
%%%%%%%%%%%%%%%%%%%%%%%%%%%%%%%%%%%%%%%%%%%%%%%%%%%%%%%%%%%%%%%%%%%%%%%%%%%%%%%%%%%%%%%%%%%%%%%%%%%%%%%%%
\section{Proofs for Section~\ref{sec:our approach}}\label{sec:proof}
%%%%%%%%%%%%%%%%%%%%%%%%%%%%%%%%%%%%%%%%%%%%%%%%%%%%%%%%%%%%%%%%%%%%%%%%%%%%%%%%%%%%%%%%%%%%%%%%%%%%%%%%%%%%

\subsection{Proof of Theorem \ref{thm:generalTheorem}}

The proof of Theorem \ref{thm:generalTheorem} requires several auxiliary results.
We first connect the characteristic function of the 1-dimensional marginals of any process $X$, which provides a stochastic solution to \eqref{eq:EvEq}, to a family of Volterra equations of second kind.

\begin{proposition}[General Relation]\label{thm:generalRelation}
Let Assumptions~\ref{ass:u_0+psi},~\ref{ass:r} hold. Then $(X_t)_{t\geq 0}$ provides a stochastic solution to the evolution equation~\eqref{eq:EvEq} if and only if
\begin{align}\label{eq:generalRelation}
1-\ffi_{X_t}(p)=\psi(p)\int_0^t k(t,s)\ffi_{X_s}(p)ds,\qquad\forall\,\, p\in\cRd,\quad\forall\,\,t>0,
\end{align}
where $\ffi_{X_t}(p):=\eE\left[  e^{ip\cdot X_t}\right]$ is the characteristic function of $X_t$,  $t>0$.
\end{proposition}
\begin{remark}
The General Relation~\eqref{eq:generalRelation} provides an interrelation between the parameters of the equation $\psi$, $k$ and the stochastic process $X$.
%In the second part of this Section, we investigate under which conditions one could choose $X$ to be a randomely scaled Gaussian process.
\end{remark}

\begin{remark}\label{ass:psi}
If relation~\eqref{eq:generalRelation} holds for each $p\in\cRd$ then it holds in particular for $p=0$. Since $\ffi_{X_t}(0)=1$ for each $t\geq0$, the right hand side of~\eqref{eq:generalRelation} must be zero at $p=0$ for each $t\geq0$. It is possible only if   $\psi(0)=0$, i.e. there is no killing term in the L\'{e}vy-Khintchine representation of $\psi$ (and an underlying L\'{e}vy process has an infinite life time).
\end{remark}

\begin{proof}[Proof of Proposition~\ref{thm:generalRelation}]
Let $u(t,x)$ be given by~\eqref{eq:u(t,x)}.   Then  $\|u(t,\cdot)\|_\infty\leq\|u_0\|_\infty$ for each $t\geq0$. Moreover, 
\begin{align*}
\eE\left[\int_{\cRd}|u_0(x+X_t)|dx  \right]=\eE\left[\int_{\cRd}|u_0(y)|dy  \right]=\|u_0\|_{L^1(\cRd)}<\infty
\end{align*}
since $u_0\in S(\cRd)\subset L^1(\cRd)$. Therefore, we have by the Fubini theorem
\begin{align*}
\int_{\cRd}\!|u(t,x)|dx\leq \int_{\cRd}\!\eE\left[|u_0(x+X_t) | \right]  dx=\eE\left[\int_{\cRd}\!|u_0(x+X_t)|dx  \right]=\|u_0\|_{L^1(\cRd)}<\infty,
\end{align*}
i.e. $u(t,\cdot)\in L^1(\cRd)$ for all $t\geq0$ and $\|u(t,\cdot)\|_{L^1(\cRd)}\leq\|u_0\|_{L^1(\cRd)}$. Hence we can apply Fourier transform to $u(t,\cdot)$ with respect to the space variable. And we have by the Fubini theorem:
\begin{align}\label{eq:Fourier-1}
\mathfrak{F}\left[u(t,\cdot)  \right](p)&=(2\pi)^{-d/2}\int_{\cRd}e^{-ix\cdot p}\eE\left[u_0(x+X_t)  \right]dx\nonumber\\
&
=\eE\left[(2\pi)^{-d/2}\int_{\cRd}e^{-i(y-X_t)\cdot p}u_0(y)dy  \right]=\mathfrak{F}[u_0](p)\ffi_{X_t}(p).
\end{align}
Note that since $u(t,\cdot)\in L^1(\cRd)$, we have $\mathfrak{F}\left[u(t,\cdot)  \right]\in C_\infty(\cRd)$ for all $t\geq0$. % \yk{(hence the GR holds for all $p\in\cRd$, not for almost all)}.
   Since $|\ffi_{X_t}(p)|\leq 1$ for all $p\in\cRd$, $t\geq0$, we have
\begin{align}\label{eq:Fourier-estimate}
\left| \mathfrak{F}[u(t,\cdot)](p)  \right|\leq \left| \mathfrak{F}[u_0](p)  \right|\qquad\forall\,\,p\in\cRd,\quad t\geq0.
\end{align}
Since $u_0\in S(\cRd)$ then also $\mathfrak{F}[u_0]\in S(\cRd)$. Hence $\mathfrak{F}[u(t,\cdot)]$ is a bounded function which decays as fast as $\mathfrak{F}[u_0]$ when $|p|\to\infty$ by~\eqref{eq:Fourier-estimate}. Therefore, $u(t,\cdot)$ is a smooth function for each $t\geq0$ by the properties of the Fourier transform. Moreover  $\lim_{|x|\to\infty}u(t,x)=0$  for all $t\geq0$ by the Lebesgue theorem on dominated convergence. Hence $u(t,\cdot)\in C_\infty(´\cRd)$.  Further, again by~\eqref{eq:Fourier-estimate}, we have   $-\psi\mathfrak{F}[u(t,\cdot)]\in L^1(\cRd)$ and is a bounded function too, since a symbols grows at most quadratically. Hence $\mathfrak{F}^{-1}\circ \psi\circ\mathfrak{F}[u(t,\cdot)] \in C_\infty(\cRd)$ by properties of the Fourier transform.  So, $u(t,\cdot)\in\Dom(L)$ for all $t\geq0$.

The integral $\int_0^t k(t,s)\ffi_{X_t}(p) ds$ is  finite for all $t\geq0$ due to Assumption~\ref{ass:r} since $|\ffi_{X_t}(p)|\leq 1$ for all $p\in\cRd$ and $t\geq0$.  Analogously,  the integrals $\int_0^t k(t,s)Lu(s,x)ds$ and $\int_0^t k(t,s)\psi(p)\mathfrak{F}[u(s,\cdot)](p)\,ds$ are well-defined  and finite for all $t\geq0$ by Assumption~\ref{ass:r}, equality~\eqref{eq:Fourier-1} and estimate~\eqref{eq:Fourier-estimate} since $-\psi \mathfrak{F}[u_0]\in L^1(\cRd)$ and is a bounded function.   Therefore, both sides of equation~\eqref{eq:EvEq} make sense for $u$ given by~\eqref{eq:u(t,x)}.  Further, it holds by Fubini theorem (since $\left|\psi\mathfrak{F}[u(t,\cdot)]\right|$ is bounded and decays fast at infinity) and by~\eqref{eq:Fourier-1}
\begin{align*}
&\int_0^t k(t,s)Lu(s,x)ds=\int_0^t k(t,s)\mathfrak{F}^{-1}\left[(-\psi)\mathfrak{F}[u(s,\cdot)]\right](x)ds\\
&
=\mathfrak{F}^{-1}\left[\int_0^t k(t,s)(-\psi)\mathfrak{F}[u(s,\cdot)]ds\right](x)=\mathfrak{F}^{-1}\left[(-\psi)\mathfrak{F}[u_0]\int_0^t k(t,s)\ffi_{X_s}ds\right](x).
\end{align*}

Assume now that $X$ is such that $u$ given by~\eqref{eq:u(t,x)} solves equation~\eqref{eq:EvEq}.  Applying Fourier tranform to both sides of equation~\eqref{eq:EvEq} we obtain by~\eqref{eq:Fourier-1} for $p\in\cRd$, $t\geq0$
\begin{align}\label{eq:just eq}
\mathfrak{F}[u_0](p)\ffi_{X_t}(p)  =\mathfrak{F}[u_0](p)-\psi(p)\mathfrak{F}[u_0](p)\int_0^t k(t,s)\ffi_{X_s}(p)ds.
\end{align}
Since $u_0$ can be arbitrary function from $S(\cRd)$, the above equality~\eqref{eq:just eq} is equivalent to~\eqref{eq:generalRelation}.  Vice versa, if $X$ is such that $\varphi_{X_t}$  solves equation~\eqref{eq:generalRelation}, then  equation~\eqref{eq:just eq} holds for any $u_0\in S(\cRd)$. And hence $u$ given by~\eqref{eq:u(t,x)} solves equation~\eqref{eq:EvEq}.
\end{proof}

We now discuss the family of Volterra equations in the General Relation 
on the space $B_b([0,T],\cC)$ of bounded Borel-measurable complex-valued functions defined on the segment $[0,T]$, $T>0$. It is a Banach space with the supremum-norm $\|\cdot\|_\infty$.

\begin{lemma}\label{lem:contraction}
Let Assumption~\ref{ass:r} hold.  Then for any $\lambda\in\cC$ and any $T>0$ there exists $n_T\in\Nat$ such that the $n_T$-th power of the  operator $\mathcal{R}_\lambda\,:\,B_b([0,T],\cC)\to B_b([0,T],\cC)$,%\ykm{besser $C((0,\infty))$ statt $L^\infty$? Dann: passende Ass 2.1 (ii)} 
\begin{align}\label{eq: operator R-lambda}
\left(\mathcal{R}_\lambda g\right)(t):=
\left\{\begin{array}{ll}
1-\lambda\int_0^t k(t,s)g(s)ds, & t\in(0,T],\\
1, & t=0,
\end{array}\right.\qquad g\in B_b([0,T],\cC).
\end{align}
is a strict contraction.
\end{lemma}

\begin{proof}
Let us fix $\lambda\in\cC$ and $T>0$.  Due to Assumption~\ref{ass:r}, it holds for any $g\in B_b([0,T],\cC)$,
$n\in \mathbb{N}$, and any $t\in(0,T]$
\begin{align}\label{eq:estimate k}
& t^{(\alpha^*-1)n}\left| \int_0^tk(t,s)g(s)ds  \right|\nonumber \\ \leq& 
\left(\sup_{0<s\leq T} |s^{(\alpha^*-1)(n-1)}g(s)|\right)
 t^{(\alpha^*-1)n} \int_0^t|k(t,s)|s^{(1-\alpha^*)(n-1)}ds\nonumber \\ \leq &   \left(\sup_{0<s\leq T} |s^{(\alpha^*-1)(n-1)}g(s)|\right)
\frac{\|k(t,\cdot)\|_{L^{1+\eps}([0,t])}t^{\alpha^*-1}t^{\eps/(1+\eps)}}{((1-\alpha^*)(n-1)(1+1/\eps)+1)^{\eps/(1+\eps)}} \nonumber \\
\leq & \frac{K_T}{((1-\alpha^*)(n-1)(1+1/\eps)+1)^{\eps/(1+\eps)}}  \left(\sup_{0<s\leq T} |s^{(\alpha^*-1)(n-1)}g(s)|\right).
\end{align}
Choosing $n=1$, we have $\| \mathcal{R}_\lambda g\|_\infty\leq 1+ K_T|\lambda| T^{1-\alpha^*}\|g\|_\infty<\infty$, i.e. the operator $\mathcal{R}_\lambda$ maps $B_b([0,T],\cC)$ into itself.
%\ykm{What about measurability of $\mathcal{R}_\lambda g$ (hence measurability of $k(\cdot,\cdot)$)?}.
 Since for all $f$, $g\in B_b([0,T],\cC)$ and all $t\in[0,T]$
\begin{align*}
\left| \mathcal{R}_\lambda g(t)-\mathcal{R}_\lambda f(t)\right|\leq|\lambda| K_T T^{1-\alpha^*}\|g-f\|_\infty,
\end{align*}
  $\mathcal{R}_\lambda$ is a continuous operator on $B_b([0,T],\cC)$. Further, for each $n\in\Nat$ and $t\in(0,T]$, we obtain
  due to \eqref{eq:estimate k}
 \begin{align*}
 & t^{(\alpha^*-1)(n+1)} \left| \mathcal{R}^{n+1}_\lambda g(t)-\mathcal{R}^{n+1}_\lambda f(t)\right|
 \\ =& t^{(\alpha^*-1)(n+1)} |\lambda| \left|\int_0^t k(t,s) (\mathcal{R}^{n}_\lambda g(s)-\mathcal{R}^{n}_\lambda f(s)) ds\right|
 \\ \leq &  \frac{K_T|\lambda|}{((1-\alpha^*)n(1+1/\eps)+1)^{\eps/(1+\eps)}}  \sup_{0<s \leq T} |s^{(\alpha^*-1)n}
 (\mathcal{R}^{n}_\lambda g(s)-\mathcal{R}^{n}_\lambda f(s))|.
 \end{align*}
 Proceeding inductively, we arrive at
  \begin{align}\label{eq:estimate 2}
&\sup_{0\leq t\leq T} \left| \mathcal{R}^{n+1}_\lambda g(t)-\mathcal{R}^{n+1}_\lambda f(t)\right| \nonumber \\ \leq & 
(|\lambda|K_T T^{1-\alpha^*})^{n+1}\prod_{l=1}^n \big(l(1-\alpha^*)(1+1/\eps)+1  \big)^{\frac{-\eps}{1+\eps}}
\|g-f\|_\infty \;    .
\end{align}
Since the factor in front of  $\|g-f\|_\infty$ tends to $0$ as $n\to\infty$, there exists $n_T\in\Nat$ such that  $\mathcal{R}^{n_T}_\lambda$ is a strict contraction on $B_b([0,T],\cC)$.
\end{proof}

\begin{corollary}\label{cor:corollary}
Let Assumption~\ref{ass:r} hold. Then, for each $\lambda\in\cC$, there exists a unique solution $\Phi(\cdot,-\lambda)\in B_b([0,T],\cC)$, $\forall\,\, T>0$, %\ykm{We need  additional assumptions on $k$ and  the Banach space $C[0,T]$ instead of $L^\infty([0,T])$ to eliminate the word ``almost'' - we formulate General Relation without ``almost''} 
 of the following Volterra  equation of the second kind
\begin{align}\label{eq:Volterra for Phi}
\Phi(t,-\lambda)=1-\lambda\int_0^t k(t,s)\Phi(s,-\lambda)ds,\qquad t>0.
\end{align}
Moreover, $\lim_{t\searrow0}\Phi(t,-\lambda)=1$ locally uniformly with respect to $\lambda\in\cC$,
 $\Phi(t,\cdot)$ is an entire function for  all $t\geq0$ and equalities~\eqref{eq:Phi} and~\eqref{eq:coeff of Phi} hold. 
\end{corollary}

\begin{proof}
Fix any $T>0$.  By the Banach fix-point theorem, there exists exactly one fixed point $\Phi(\cdot,-\lambda)\in B_b([0,T],\cC)$  of the  strict contraction $\mathcal{R}_\lambda^{n_T}$  
due to Lemma~\ref{lem:contraction}. Hence $\Phi(\cdot,-\lambda)$ is also the unique fixed point of the operator $\mathcal{R}_\lambda$, i.e. the equation~\eqref{eq:Volterra for Phi} has the unique solution $\Phi(\cdot,-\lambda)\in B_b([0,T],\cC)$ which can be obtained by the Picard iterations %\ykm{the limit is in ess-sup-norm}
\begin{align*}
\Phi(t,-\lambda)=\lim_{n\to\infty}\Phi_n(t,-\lambda), \qquad t\in[0,T], 
\end{align*}
where %\ykm{all identities are in $L^\infty$-sense}
\begin{align*}
&\Phi_0(t,-\lambda):=1,\quad t\in[0,T],\\
&
\Phi_n(t,-\lambda):=\left(\mathcal{R}_\lambda\Phi_{n-1}(\cdot,-\lambda)\right)(t)=\\
&
\phantom{\Phi_n(t,-\lambda):}
\!=
\left\{
\begin{array}{ll}
1-\lambda\int_0^t k(t,s)\Phi_{n-1}(s,-\lambda)ds, &  t\in(0,T], \\
1, & t=0, 
\end{array}\right. \quad n\in\Nat.
\end{align*}
Using auxilliary functions $\phi_0(t,\lambda):=1$ and $\phi_n(t,\lambda):=\Phi_n(t,-\lambda)-\Phi_{n-1}(t,-\lambda)$
 for $n\in\Nat$ (i.e. $\phi_n(t,\lambda)= -\lambda\int_0^t k(t,s)\phi_{n-1}(s,\lambda)ds$ for $t\in(0,T]$), we get
$$
\Phi(t,-\lambda)=\sum_{n=0}^\infty\phi_n(t,\lambda)=\sum_{n=0}^\infty c_n(t)(-\lambda)^n,
$$
where the coefficients $c_n(t)$, $n\in\Nat$, $t\in[0,T]$, are given by~\eqref{eq:coeff of Phi}. Therefore, for each $\lambda\in\cC$ and each $T>0$, the function $\Phi$, given by~\eqref{eq:Phi}, belongs to $B_b([0,T],\cC)$. 
Therefore, for  all $t\geq0$, the series $\sum_{n=0}^\infty c_n(t)\lambda^n$ converges in $\cC$, i.e. the function $\Phi(t,\cdot)$ is an entire function for  all $t\geq0$.
Moreover, we have for each $n\in\Nat$ and each $t\in[0,T]$, by iterating ~\eqref{eq:estimate k} analogously to the derivation of~\eqref{eq:estimate 2},
\begin{align*}
|c_n(t)|\leq K_T^{n}T^{n(1-\alpha^*)}\prod_{l=1}^{n-1} \left[ \big(l(1-\alpha^*)(1+1/\eps)+1  \big)^{\frac{-\eps}{1+\eps}} \right]\to0,\quad T\to0,
\end{align*}
and for any $R>0$ and any $\lambda\in\{z\in\cC\,:\,|z|\leq R\}$
\begin{align*}
\left| \Phi(t,-\lambda)  \right|&\leq\sum_{n=0}^\infty R^n|c_n(t)|\\
&
\leq\sum_{n=0}^\infty R^n K_T^{n} T^{n(1-\alpha^*)}\prod_{k=1}^{n-1} \left[ \big(k(1-\alpha^*)(1+1/\eps)+1  \big)^{\frac{-\eps}{1+\eps}} \right]<\infty.
\end{align*}
Hence, we have
$\liml_{t\searrow0}|\Phi(t,-\lambda)|=1$   and $\liml_{t\searrow0}\Phi(t,-\lambda)=1$  locally uniformly w.r.t. $\lambda\in\cC$.
\end{proof}

We are now ready to present  the proof of Theorem \ref{thm:generalTheorem}.
\begin{proof}[Proof of Theorem \ref{thm:generalTheorem}]
 (i) By Corollary \ref{cor:corollary}, the family of Volterra equations for the General Relation
~\eqref{eq:generalRelation} has $\Phi(t,-\psi(p))$ as its unique locally bounded solution, where $\Phi$ is given by~\eqref{eq:Phi},~\eqref{eq:coeff of Phi}.
 Thus, by Proposition \ref{thm:generalRelation}, $(X_t)_{t\geq 0}$ provides a stochastic solution to the evolution equation~\eqref{eq:EvEq}, if and only if $\ffi_{X_t}(\cdot)=\Phi(t,-\psi(\cdot))$ holds for every $t\geq 0$. Now, if a stochastic solution exists, then 
 $\Phi(t,-\psi(\cdot))$ is positive definite for every $t\geq 0$, since every characteristic function has this property. 
 On the other hand, if  $\Phi(t,-\psi(\cdot))$ is  a positive definite function, then, for each $t\geq0$, there exists a random variable $\tilde X_t$ such that $\Phi(t,-\psi(\cdot))=\ffi_{\tilde X_t}$ (recalling that $\Phi(t,0)=1$).
We may now choose any stochastic process $(X_t)_{t\geq0}$ such that  $X_t$ has the same distribution as $\tilde X_t$ for every $t\geq 0$. E.g. one can construct an independent family $(X_t)_{t\geq 0}$ with these one-dimensional marginal distributions
on an infinite product space. Then, $\ffi_{X_t}(\cdot)=\Phi(t,-\psi(\cdot))$ for every $t\geq 0$, and, hence, a stochastic solution exists.
\\[0.2cm]
(ii) Let the restriction of $\Phi(t,-\cdot)$ on $(0,\infty)$ be a completely monotone function for each $t\geq0$. Hence, for each $t\geq0$, there exists a non-negative random variable $A(t)$ whose distribution $\mathcal{P}_{A(t)}$ has Laplace transform $\Phi(t,-\cdot)$ by the Bernstein theorem. Let $(Y_t)_{t\geq0}$ be a $d$-dimensional L\'{e}vy process with characteristic exponent $\psi$ which is independent from $(A(t))_{t\geq0}$. Then we have
\begin{align*}
\Phi(t,-\psi(p))&=\int_0^\infty e^{-a\psi(p)}\mathcal{P}_{A(t)}(da)=\int_0^\infty  \eE\left[ e^{ip\cdot Y_a}\right]\mathcal{P}_{A(t)}(da)\\
&
=\eE\left[ e^{ip\cdot Y_{A(t)}}\right]=\ffi_{Y_{A(t)}}(p).
\end{align*}
Therefore, the function $\Phi(t,-\psi(\cdot))$ is positive definite and  $(Y_{A(t)})_{t\geq 0}$ provides a stochastic solution to \eqref{eq:EvEq} by statement~(i) of Theorem~\ref{thm:generalTheorem}.  
\\[0.2cm]
(iii)
Let now, additionally, the symbol $\psi$ be given by $\psi:=f\circ\widetilde{\psi}$ for some other CNDF $\widetilde{\psi}$ and some Bernstein function $f$. Then the function $\Phi(t,-f(\cdot))$ is completely monotone as a composition CMF$\circ$BF. 
Hence, for each $t\geq0$, there exists a non-negative random variable $\widetilde{A}(t)$ whose distribution $\mathcal{P}_{\widetilde{A}(t)}$ has Laplace transform $\Phi(t,-f(\cdot))$. Taking a $d$-dimensional L\'{e}vy process $(\widetilde{Y}_t)_{t\geq0}$ with characteristic exponent $\widetilde{\psi}$ which is independent from $(\widetilde{A}(t))_{t\geq0}$, we obtain
\begin{align*}
\Phi(t,-\psi(p))&=\Phi(t,-f(\widetilde{\psi}(p)))  =\int_0^\infty e^{-a\widetilde{\psi}(p)}\mathcal{P}_{\widetilde{A}(t)}(da)\\
&
=\int_0^\infty  \eE\left[ e^{ip\cdot \widetilde{Y}_a}\right]\mathcal{P}_{\widetilde{A}(t)}(da)=\eE\left[ e^{ip\cdot \widetilde{Y}_{\widetilde{A}(t)}}\right]=\ffi_{\widetilde{Y}_{\widetilde{A}(t)}}(p),
\end{align*}
and we may conclude as in (ii).
 \end{proof}

%\begin{remark}
%Note that  $\Phi(t,0)=1$ for  all $t\geq0$ by~\eqref{eq:Phi} and~\eqref{eq:coeff of Phi}, as well as $\Phi(0,-\lambda)=1$ for all $\lambda\in\cC$ by the estimate~\eqref{eq:estimate k}.
%\end{remark}

%\begin{remark}
%If $\Phi(t,-\cdot)$ is CM, General Relation has the following form ........... +scaling condition
%\end{remark}

\begin{remark}
Let $\Phi$ and $f$ be as in Theorem~\ref{thm:generalTheorem}~(iii).  Then the completely monotone function $\Phi(t,-f(\cdot))$ extends to an analytical function in $\Pi_+:=\{z\in\cC\,:\,\Re z>0\}$. This function does not have to be analytical in the whole complex plane. It follows from Corollary~\ref{cor:corollary} (with $\lambda=f(z)$) that $ \widetilde{\Phi}(t,\cdot):= \Phi(t,-f(\cdot))$ solves the following version of Volterra equation~\eqref{eq:Volterra for Phi}:
\begin{align*}
\widetilde{\Phi}(t,z)=1-f(z)\int_0^tk(t,s)\widetilde{\Phi}(s,z)ds,\qquad t>0,\quad z\in\Pi_+.
\end{align*}
\end{remark}

\begin{remark}
 Suppose that Assumptions \ref{ass:u_0+psi}, \ref{ass:r} are in force. If $(X_t)_{t\geq 0}$ provides a stochastic solution to \eqref{eq:EvEq}, then by Theorem 
 \ref{thm:generalTheorem} and \eqref{eq:Fourier-1},
 $$
 \mathbb{E}\left[u_0\left( x+X_t\right)\right]=\mathfrak{F}^{-1}\big[\mathfrak{F}[u_0](\cdot)\Phi(t,-\psi(\cdot))\big](x).
 $$
 The function on the right-hand side may be well-defined, even if no stochastic solutions exists. Adapting the arguments in the proof of Proposition \ref{thm:generalRelation} in the obvious way, one can e.g. show that
 $$
 u(t,x):=\mathfrak{F}^{-1}\big[\mathfrak{F}[u_0](\cdot)\Phi(t,-\psi(\cdot))\big](x)
 $$
 solves \eqref{eq:EvEq} for every initial condition $u_0\in S(\cRd)$, if $\Phi(t,-\psi(p))$
 satisfies a polynomial growth condition in $p$ which is locally uniform in $t$.
\end{remark}

\subsection{Proof of Proposition \ref{prop:cont}}
Continuity of $c_n$ at $t=0$ has already been shown in the proof
 of Corollary 3.1. We next prove inductively that $c_n$ is continuous on $(0,\infty)$. The claim is trivial for the constant function $c_0$. We write
 $$
 c_n(t)=t\int_0^1 k(t,ts)c_{n-1}(st) ds,
$$
and show that $c_n$ is continuous on $(1/T,T)$ for every $T>0$.
In view of the induction hypothesis and the continuity assumption on $k$, we only need to argue that we can interchange limit (in the $t$-variable) and integration in this expression. To this end, we apply the de la Val\'ee-Poussin criterion for uniform integrability of the family $(k(t,ts)c_{n-1}(st))_{t\in (1/T,T)}$ for arbitrary $T>0$. Thanks to Assumption 2.1, with the same choice of $\eps, \alpha^*$ as there,
\begin{eqnarray*}
&& \sup_{1/T\leq t\leq T} \int_0^1 |k(t,ts)c_{n-1}(st)|^{1+\eps} ds\\ & \leq& \sup_{0\leq u \leq T} |c_{n-1}(u)|^{1+\eps}  \sup_{1/T\leq t \leq T} t^{-1}t^{\alpha^*(1+\eps)} \int_0^t |k(t,s)|^{1+\eps}ds\;T^{\alpha^*(1+\eps)} \\
&\leq & T^{\alpha^*(1+\eps)} \sup_{0\leq u \leq T} |c_{n-1}(u)|^{1+\eps} K_T^{1+\eps}< \infty.
\end{eqnarray*}
This argument finishes the proof of continuity of $c_n$. Recall that
$$
\Phi(t,\lambda)=\sum_{n=0}^\infty c_n(t)\lambda^n.
$$
We have shown that all the summands are continuous in $(t,\lambda)$. Now, for every $T>0$, $t\in [0,T]$ and $\lambda\in \mathbb{C}$ with $|\lambda|\leq T$,
$$
|c_n(t)\lambda^n|\leq K_T^{n}T^{n(2-\alpha^*)}\prod_{l=1}^{n-1} \left[ \big(l(1-\alpha^*)(1+1/\eps)+1  \big)^{\frac{-\eps}{1+\eps}} \right],
$$
by the proof of Corollary 3.1,
and the right-hand side is summable. We may thus interchange limits in the $(t,\lambda)$-variables and summation,  yielding the continuity of $\Phi$.
For the continuity of the function $u$, recall that by \eqref{eq:Fourier-1}
$$
u(t,x)=\mathfrak{F}^{-1}[\mathfrak{F}[u_0](\cdot)\varphi_{X_t}(\cdot)](x),
$$
and thus continuity is inherited from $\Phi$ as a
consequence of the dominated convergence theorem with integrable majorant $|\mathfrak{F}[u_0]|$, since  characteristic functions are bounded by 1.

\subsection{Proof of Theorem \ref{thm:hom}}

Suppose that $k$ is homogeneous of degree $\beta-1$ for some $\beta>0$.
Since $k(t,ts)=t^{\beta-1}k(1,s)$, Assumption \ref{ass:cont} is obviously satisfied. We next check Assumption \ref{ass:r}. To this end note that
\begin{eqnarray*}
 t^{-1/(1+\eps)} \| k(t,\cdot)\|_{L^{1+\eps}((0,t))} &=& t^{-1/(1+\eps)}\left(t \int_0^1 |k(t,ts)|^{1+\eps} ds\right)^{1/(1+\eps)}\\ &=& t^{\beta-1} \|k(1,\cdot)\|_{L^{1+\eps}((0,1))}.
\end{eqnarray*}
Hence, Assumption \ref{ass:r} is satisfied with the choice $\alpha^*:=1-\beta\in [0,1)$  in the case $\beta\in(0,1]$ and with the choice $\alpha^*:=0$ in the case $\beta>1$.

We next observe, inductively, that
$c_n(t)=c_n(1)t^{n\beta}$, because
\begin{align*}
c_n(t)&=\int_0^t k(t,s)c_{n-1}(s)ds=t^{\beta}\int_0^1k(1,s)c_{n-1}(ts)ds\\
&
=t^{n\beta}\int_0^1k(1,s)c_{n-1}(s)ds=t^{n\beta}c_n(1).
\end{align*}
Let $\hat c_n=c_n(1)$. Then, by the previous considerations, $\hat c_n$ satsifies the recursion
$$
\hat c_n= \hat c_{n-1} \int_0^1k(1,s)s^{\beta(n-1)} ds,\quad \hat c_0=1
$$
and
\begin{align*}
\Phi(t,\lambda)=\sum_{n=0}^\infty \hat c_n (t^\beta\lambda)^n=\hat \Phi(t^\beta\lambda).
\end{align*}
Assume now that the restriction of the function $\hat \Phi(-\cdot)$ on $(0,\infty)$ is completely monotone. Then, for 
$\gamma\in(0,1]$, $x \mapsto \hat\Phi(-x^\gamma)$ is completely monotone on $(0,\infty)$, because $(\cdot)^\gamma$
is a Bernstein function. Thus, there is a nonnegative random variable $\tilde A$ with Laplace transform given by 
$\hat\Phi(-(\cdot)^\gamma)$. Then, for every $t\geq 0$, the Laplace transform of $\tilde A t^{\beta/\gamma}$ is given by
$$
\mathbb{E}\left[ e^{-\lambda \tilde A t^{\beta/\gamma}} \right]=\hat\Phi(-\lambda^\gamma t^\beta)=\Phi(t,-\lambda^\gamma),\quad \lambda>0.
$$
Now, Theorem \ref{thm:generalTheorem}, (iii), applies.

\subsection{Proof of Theorem \ref{thm:conv}}

We first note that Assumption \ref{ass:r} is satisfied  with $\alpha^*=1-\beta\in [0,1)$ for 
any sufficiently small $\eps>0$, because
\begin{eqnarray*}
&& \sup\limits_{0<t\leq T}t^{\alpha^* -\frac{1}{1+\eps}}\|k(t,\cdot)\|_{L^{1+\eps}((0,t))}\\ &\leq& Me^{\gamma T} \sup\limits_{0<t\leq T}\left(t^{\alpha^*-\frac{1}{1+\eps}} \left(\int_0^t (t-s)^{-\alpha^*(1+\eps)} ds\right)^{1/(1+\eps)}\right)=\frac{Me^{\gamma T}}{(1-\alpha^*(1+\eps))^{1/(1+\eps)}}.
\end{eqnarray*}

We  next prove inductively that
 $$
 e^{-\gamma t} |c_n(t)| \leq (M \Gamma(\beta) t^\beta)^n \frac{1}{\Gamma (n\beta +1)}.
 $$
 This is obvious for $n=0$ and the induction step follows by 
 \begin{eqnarray*}
  e^{-\gamma t} |c_n(t)| &\leq& \int_0^t M(t-s)^{\beta-1} e^{-\gamma s} |c_{n-1}(s)|ds
  \\ &\leq& M^n  \Gamma(\beta)^{n-1} \frac{1}{\Gamma ((n-1)\beta +1)}\int_0^t (t-s)^{\beta-1} s^{\beta(n-1)} ds \\ &=& 
   M^n  \Gamma(\beta)^{n-1} t^{\beta n} \frac{1}{\Gamma ((n-1)\beta +1)}\int_0^1 (1-s)^{\beta-1} s^{\beta(n-1)} ds
   \\ &=& (M \Gamma(\beta) t^\beta)^n \frac{1}{\Gamma (n\beta +1)}.
 \end{eqnarray*}
Then, for every $\lambda\in \mathbb{C}, \;t\geq 0$,
\begin{eqnarray*}
 \sum_{n=0}^\infty |c_n(t)| |\lambda|^n\leq e^{\gamma t} E_\beta(M\Gamma(\beta)|\lambda| t^\beta)\leq const.\; e^{(\gamma+
 (M\Gamma(\beta)|\lambda|)^{1/\beta})t}
\end{eqnarray*}
using the asymptotics for the Mittag-Leffler function, which can be found e.g. in \cite{MR2800586}, Eq. (6.4). We conclude that for every $\lambda\geq0$ and 
$\sigma> \gamma+
 (M\Gamma(\beta)|\lambda|)^{1/\beta}$, the Laplace  transform of 
 $\Phi(\cdot,-\lambda)$ exists and can be interchanged with the summation (by Fubini's theorem with Lebesgue measure and counting measure), i.e.,
 \begin{eqnarray*}
  (\mathcal{L}\Phi(\cdot,-\lambda))(\sigma)=\sum_{n=0}^\infty 
  (\mathcal{L}c_n)(\sigma)(-\lambda)^n.
 \end{eqnarray*}
By the convolution theorem for the Laplace transform and induction
$$
(\mathcal{L}c_n)(\sigma)=(\mathcal{L}\mathfrak{K})(\sigma)(\mathcal{L}c_{n-1})(\sigma)=\frac{1}{\sigma}[(\mathcal{L}\mathfrak{K})(\sigma)]^n
$$
for $\sigma>\gamma$, since $(\mathcal{L}1)(\sigma)=\sigma^{-1}$. Thus, for $\sigma> \gamma+
 (M\Gamma(\beta)|\lambda|)^{1/\beta}$,
 \begin{eqnarray*}
  (\mathcal{L}\Phi(\cdot,-\lambda))(\sigma)=\frac{1}{\sigma}\frac{1}{1+\lambda(\mathcal{L}\mathfrak{K})(\sigma)}.
 \end{eqnarray*}
 Since $\Phi$ inherits continuity from $\mathfrak{K}$ by Proposition \ref{prop:cont}
 and $t\mapsto E\left[e^{-\lambda A(t)}\right]$ is RCLL by dominated convergence, Lerch's uniqueness theorem implies that
 $$
 \Phi(t,-\lambda)= E\left[e^{-\lambda A(t)}\right]
 $$
 for every $\lambda\geq0$ and $t\geq 0$. Hence, $\Phi(t,-\cdot)$ is CM for every $t\geq 0$ and part (ii) of Theorem \ref{thm:generalTheorem} applies for the assertion concerning the stochastic solution.

%%%%%%%%%%%%%%%%%%%%%%%%%%%%%%%%%%%%%%%%%%%%%%%%%%%%%%%%%%%%%%%%%%%%%%%%%%%%%%%%%%%%%%%%%%%%%%%%%%%%%%%%%%%%%%%%
\section{Stochastic solutions for generalized time-fractional evolution equations with Saigo-Maeda operators}\label{sec:special case}
%%%%%%%%%%%%%%%%%%%%%%%%%%%%%%%%%%%%%%%%%%%%%%%%%%%%%%%%%%%%%%%%%%%%%%%%%%%%%%%%%%%%%%%%%%%%%%%%%%%%%%%%%%%%%%%%%%%
In this  section, we consider generalized time-fractional evolution equations of the form~\eqref{eq:EvEq}, where the kernel $k$ is the kernel of some Saigo-Maeda operator of generalized fractional calculus. Saigo-Maeda operators provide extensions of the well-known operators of fractional calculus and include the Riemann-Liouville, Weil, Erd\'{e}lyi-Kober and Saigo operators as special cases.  We construct the function $\Phi$ correspondig  to such kernel $k$ and discuss stochastic solutions in terms of randomly slowed-down L\'evy processes of the considered evolution equations.

%a particular class of evolution equations~\eqref{eq:EvEq} which generalizes the governing equation for GGBM~\eqref{GGBM-eq}. We provide corresponding  solutions of General Relation~\eqref{eq:generalRelation} and discuss corresponding stochastic processes emerging in Theorem~\ref{thm:generalTheorem}. 

Let us first recall that  Appell's third generalization $F_3$  of the Gauss hypergeometric function is defined in the following way:
\begin{align}\label{eq:F3}
F_3\left(\alpha,\alpha',\beta,\beta',\gamma,x,y\right)=\sum_{m,n\geq 0} \frac{(\alpha)_m(\beta)_m(\alpha')_n(\beta')_n}{(\gamma)_{m+n} n!m!} x^my^n, 
\end{align}
where $\alpha,\alpha',\beta,\beta',\gamma\in\cC$, $\gamma \notin -\mathbb{N}$,    and    the general Pochhammer symbol $(\lambda)_\nu$ is defined as follows: %$(0)_0:=1$ and
\begin{align*}
(\lambda)_\nu:=\left\{\begin{array}{lll}
1, & \nu=0, & \lambda\in\cC\\%\setminus\{0\}\\
\lambda(\lambda-1)\cdot\ldots\cdot(\lambda+n-1), & \nu=n\in\Nat, & \lambda\in\cC.
\end{array}\right.
\end{align*}
The series in~\eqref{eq:F3}  converges for $|y|,|x|<1$ and can be analytically extended to reals $x,y<1$.

\begin{theorem}\label{thm:k-Phi_Saigo-Maeda}
Let  $b>0$, $a>0$, $\mu>-1$, and $\nu>\max\{-b,-a\mu\}$. Consider the kernel
\begin{align}\label{eq:k-Saigo-Maeda}
k(t,s):=\frac{a}{\Gamma(b/a)}(t^a-s^a)^{\frac{b}{a}-1}t^{a-\nu}s^{\nu-1} F_3\left(\frac{\nu}{a}-1,\frac{b}{a},1,\mu, \frac{b}{a}, 1-\left(\frac{s}{t}\right)^a,1-\left(\frac{t}{s}\right)^a\right),
\end{align}
where $0<s<t$.  Then the kernel $k$  is homogeneous of degree $b-1$ and satisfies 
$k(1,\cdot)\in L^{1+\eps}((0,1))$ for some $\eps>0$.  The corresponding function $\hat \Phi$ in Theorem~\ref{thm:hom}  has the following form:
\begin{align}\label{eq:Phi Saigo-Maeda}
 \hat \Phi(z)= \Gamma(\lambda_2) E_{\lambda_1,\lambda_2}^{\lambda_3}(z),
 \end{align}
 where
\begin{align}\label{eq:lambda 1-2-3}
\lambda_1=\frac{b}{a},\quad \lambda_2=\frac{\nu}{a}+\mu,\quad \lambda_3= 1+\frac{\nu-a}{b},
\end{align}
and  $E_{\lambda_1,\lambda_2}^{\lambda_3}$ is the three parameter Mittag-Leffler (or Prabhakar) function\footnote{The function $E_{\lambda_1,\lambda_2}^{\lambda_3}$ is well-defined on the whole $\cC$ for $\Re\lambda_1>0$ and is an entire function. }
\begin{align*}
E_{\lambda_1,\lambda_2}^{\lambda_3}(z):=\sum_{n=0}^\infty \frac{\left(\lambda_3\right)_n}{\Gamma\left(\lambda_1 n+\lambda_2\right)n!}\,z^n.
\end{align*}
\end{theorem}

\begin{proof}
First note that
\begin{align*}
k(t,ts)= \frac{a}{\Gamma(b/a)} t^{b-1} (1-s^a)^{b/a-1}s^{\nu-1} F_3(\nu/a-1,b/a,1,\mu, b/a, 1-s^a,1-1/s^a),
\end{align*}
and, thus, $k$ is homogeneous of degree $b-1$.
Let 
$$
\kappa(s):=\frac{1}{\Gamma(b/a)}  (1-s)^{b/a-1}s^{(\nu-1)/a} F_3(\nu/a-1,b/a,1,\mu, b/a, 1-s,1-1/s),\quad 0<s<1.
$$
Then, $k(t,ts)=t^{b-1}k(1,s)=t^{b-1}a\kappa(s^a)$. In particular,
 \begin{align*}
 \|k(1,\cdot)\|^{1+\eps}_{L^{1+\eps}((0,1))}&= \int_0^1 (a\kappa(s^a))^{1+\eps} ds =a^\eps \int_0^1 \kappa(s)^{1+\eps} s^{1/a-1}ds 
\end{align*}
Inserting the definition of $\kappa$, the integral converges, if and only if the integral
$$
\int_0^1 (1-s)^{(b/a-1)(1+\eps)}s^{(1+\eps)(\nu-1)/a+1/a-1} F_3(\nu/a-1,b/a,1,\mu, b/a, 1-s,1-1/s)^{1+\eps} ds
$$
does. Recall the asymptotic behavior  (see, e.g., Lemma~3.1.2 in~\cite{Saxena})
\begin{align*}
 F_3(\nu/a-1,b/a,1,\mu, b/a, 1-s,1-1/s)=O\left((1-s)^{\min\{0,1-\frac{b}{a}\}}\right),\quad s\rightarrow 1,
\end{align*}
which shows convergence at $s=1$ for small $\eps>0$, and
\begin{align*}
 F_3(\nu/a-1,b/a,1,\mu, b/a, 1-s,1-1/s)=O\left(s^{\min\{\frac{b}{a},\mu,\frac{b-\nu}{a}\}}\right),\quad s\rightarrow 0,
\end{align*}
which implies convergence at $s=0$ for small $\eps$ due to our conditions on the parameters $a$, $b$, $\mu$, $\nu$.

Let us now determine the corresponding function $\Phi$. The recursion formula for the coefficients $\hat c_n$'s (initialized 
with $\hat c_0=1)$ reads
\begin{align*}
\hat c_n&=\hat{c}_{n-1}  \int_0^1 k(1,s)s^{(n-1)b} ds= \hat{c}_{n-1}  a \int_0^1 \kappa(s^a)(s^a)^{(n-1)b/a} ds \\
&= \hat{c}_{n-1} \int_0^1 \kappa(\theta) \theta^{(n-1)b/a+1/a-1}d\theta.
\end{align*}
Then,
\begin{align*}
 \frac{\hat c_{n}}{\hat c_{n-1}}= \frac{1}{\Gamma(b/a)} \int_0^1 \theta^{\frac{nb+\nu}{a}-1} (1-\theta)^{\frac{b}{a}-1}\theta^{-\frac{b}{a}} F_3\left(\frac{\nu}{a}-1, \frac{b}{a}, 1, \mu,\frac{b}{a}, 1-\theta,1- \frac1\theta \right) d\theta.
\end{align*}
Recall that the Saigo-Maeda operator $I^{\alpha,\alpha',\beta,\beta',\gamma}_{0+}$ is defined in the following way (see, e.g., Def.~3.2.1. in~\cite{Saxena}):
\begin{align*}
&\left( I^{\alpha,\alpha',\beta,\beta',\gamma}_{0+} f\right)(x):=\\
&
\qquad=\frac{x^{-\alpha}}{\Gamma(\gamma)}\int_0^x (x-\theta)^{\gamma-1}\theta^{-\alpha'}F_3\left(\alpha,\alpha',\beta,\beta',\gamma, 1-\frac{\theta}{x}, 1-\frac{x}{\theta}   \right)f(\theta)d\theta.
\end{align*}
Hence we have with $f(\theta):=\theta^{\frac{n b+\nu}{a}-1}$
\begin{align*}
 \frac{\hat c_{n}}{\hat c_{n-1}}= \left( I^{\frac{\nu}{a}-1,\frac{b}{a}, 1,\mu,\frac{b}{a}}_{0+} f\right)(1).
 \end{align*}
Note that   the integral $I^{\alpha,\alpha',\beta,\beta',\gamma}_{0+} f$ converges for $f(\theta):=\theta^{\rho-1}$ and it holds
\begin{align*}
\left( I^{\alpha,\alpha',\beta,\beta',\gamma}_{0+} f\right)(x)=\frac{\Gamma(\rho)\Gamma(\rho+\gamma-\alpha-\alpha'-\beta)\Gamma(\rho-\alpha'+\beta')}{\Gamma(\rho+\beta')\Gamma(\rho+\gamma-\alpha-\alpha')\Gamma(\rho+\gamma-\alpha'-\beta)}
\end{align*}
in the case when  $\gamma>0$ and $\rho>\max\{0,\alpha+\alpha'+\beta-\gamma, \alpha'-\beta' \}$ (cf., e.g., Example~3.2.1 in~\cite{Saxena}). Our choice of parameters 
$$
\alpha=\frac{\nu}{a}-1,\quad \alpha'=\gamma=\frac{b}{a}, \quad \beta=1,\quad \beta'=\mu, \quad \rho=\frac{nb+\nu}{a},\;n\in\Nat,
$$
leads,   therefore, to the conditions 
\begin{equation}\label{eq:conditions_ML1}
a,b>0, \quad b>-\nu,\quad \nu>-a\mu,
\end{equation}
which we have postulated in the statement of Theorem~\ref{thm:k-Phi_Saigo-Maeda}.  Assuming \eqref{eq:conditions_ML1}, we, thus, obtain,
on the one hand,
\begin{eqnarray*}
 \frac{\hat c_{n}}{\hat c_{n-1}}&=& \frac{\Gamma(\frac{nb+\nu}{a})\Gamma(\frac{nb}{a})\Gamma(\frac{(n-1)b+\nu}{a}+\mu)}{\Gamma(\frac{nb+\nu}{a}+\mu)
 \Gamma(\frac{nb}{a}+1)\Gamma(\frac{nb+\nu}{a}-1)} = \frac{\Gamma(\frac{(n-1)b+\nu}{a}+\mu)}{\Gamma(\frac{nb+\nu}{a}+\mu)}\cdot \frac{\frac{nb+\nu}{a}-1}{\frac{nb}{a}} \\
 &=& \frac{\Gamma(\frac{(n-1)b+\nu}{a}+\mu)}{\Gamma(\frac{nb+\nu}{a}+\mu)} \cdot\frac{(n-1)+1+\frac{\nu-a}{b}}{n}.
\end{eqnarray*}
On the other hand, it holds for the coefficients of the three parameter Mittag-Leffler function $E^{\lambda_3}_{\lambda_1,\lambda_2}$: 
$$
\left(\frac{(\lambda_3)_n}{\Gamma(\lambda_1n+\lambda_2) n!}\right)/\left( \frac{(\lambda_3)_{n-1}}{\Gamma(\lambda_1(n-1)+\lambda_2) (n-1)!}\right)= 
\frac{\Gamma(\lambda_1(n-1)+\lambda_2) }{\Gamma(\lambda_1n+\lambda_2)}\cdot\frac{(n-1)+\lambda_3}{n}.
$$
Since $E^{\lambda_3}_{\lambda_1,\lambda_2}(0)=\frac{1}{\Gamma(\lambda_2)}$, we  obtain  with  $\lambda_1$, $\lambda_2$, $\lambda_3$ as in~\eqref{eq:lambda 1-2-3}
$$
\hat\Phi(z)=\Gamma(\lambda_2)E^{\lambda_3}_{\lambda_1,\lambda_2}(z),\qquad z\in\cC.
$$
\end{proof}

\begin{remark}\label{rem:CM-3parameterMLF}
Note that sufficient conditions for the complete monotonicity of the function  $z\mapsto  \Gamma(\lambda_2) E_{\lambda_1,\lambda_2}^{\lambda_3}(-z)$ are given  (see, e.g., \cite{Gorska2020} or the Appendix in~\cite{An}) by
 $$
 0<\lambda_1\leq 1,\quad 0<\lambda_3\leq \frac{\lambda_2}{\lambda_1}.
 $$
 This implies additional assumptions on parameters $a$, $b$, $\mu$, $\nu$:
 \begin{align*}
 b\leq a,\quad \nu>a-b,\quad \mu\geq \frac{b}{a}-1.
 \end{align*}
\end{remark}
Hence the following statement is a direct consequence of Theorem~\ref{thm:generalTheorem}, Theorem~\ref{thm:k-Phi_Saigo-Maeda}, Remark~\ref{rem:decoration} and Remark~\ref{rem:CM-3parameterMLF}.

\begin{corollary}\label{cor:saigo_maeda}
Let   $b>0$, $a\geq b$, $\mu\geq \frac{b}{a}-1$, $\nu>\max\left\{a-b,-a\mu   \right\}$. Let the kernel $k$ be given by~\eqref{eq:k-Saigo-Maeda}, $g\in\mathcal{G}$ from Remark~\ref{rem:decoration} and the corresponding kernel $\kappa_g$ be given by~\eqref{eq: k decorated}. Let Assumption~\ref{ass:u_0+psi} hold. 
 Let $A$ denote a non-negative random variable, whose distribution has Laplace transform  $z\mapsto  \Gamma(\lambda_2) E_{\lambda_1,\lambda_2}^{\lambda_3}(-z)$ with parameters $\lambda_1$, $\lambda_2$, $\lambda_3$ as in~\eqref{eq:lambda 1-2-3}.   Let $(Y_t)_{t\geq0}$ be a $\cRd$-valued L\'{e}vy process with generator $(L,\Dom(L))$ which is independent of $A$.  Then 
\begin{align}
v(\tau,x):=\eE\left[u_0\left( x+Y_{Ag^b(\tau)} \right)  \right]
\end{align}
solves the corresponding  generalized time-fractional evolution equation~\eqref{eq:decoratedEvEq}.
\end{corollary}

\begin{example}\label{ex:saigo_maeda}
(i) Let us consider the case  $\lambda_3=1$, i.e. $\nu=a$. Then $E^{\lambda_3}_{\lambda_1,\lambda_2}$ reduces to the two parameter Mittag-Leffler function $E_{\lambda_1,\lambda_2}$ and we have
$$
\Phi(t,z)=\Gamma(\mu+1)E_{\frac{b}{a},\mu+1}(zt^b).
$$
Further,
\begin{align*}
 F_3(0,b/a,1,\mu, b/a, 1-(s/t)^a,1-(t/s)^a)=\sum_{n=0}^\infty\frac{(\mu)_n}{n!}\left( 1-(t/s)^a \right)^n   =(s/t)^{\mu a}.
\end{align*}
And the corresponding kernel $k$ simplifies to
 $$
 k(t,s)= \frac{a}{\Gamma(b/a)}(t^a-s^a)^{b/a-1}t^{-a\mu}s^{a(\mu+1)-1}.
 $$

\noindent (ii) Let us consider the case $\lambda_2=\lambda_3=1$, i.e. $\nu=a$ and $\mu=0$. Then $E^{\lambda_3}_{\lambda_1,\lambda_2}$ reduces to the classical  Mittag-Leffler function $E_{\lambda_1}$ and we have
$$
\Phi(t,z)=E_{\frac{b}{a}}(zt^b).
$$
The corresponding kernel $k$ simplifies to
 $$
 k(t,s)= \frac{a}{\Gamma(b/a)}(t^a-s^a)^{b/a-1}s^{a-1}.
 $$
 Let now $\beta\in(0,1]$, $\alpha\in(0,2)$. Choosing $b:=\alpha$ and $a:=\frac{\alpha}{\beta}$ we obtain the kernel of the governing equation~\eqref{GGBM-eq} of the GGBM and $\Phi(t,z)=E_\beta(z t^\alpha)$. If additionally $\alpha=\beta$, we get $k(t,s)=\frac{1}{\Gamma(\beta)}(t-s)^{\beta-1}$ and $\Phi(t,z)=E_\beta(z t^\beta)$, cp. Example \ref{ex:fractional}.
\end{example}

\begin{remark}
It follows immediately from Theorem~\ref{thm:generalTheorem}, Theorem~\ref{thm:k-Phi_Saigo-Maeda} and Corollary~\ref{cor:corollary}, that the three parameter Mittag-Leffler function $E^{1+\frac{\nu-a}{b}}_{\frac{b}{a},\,\frac{\nu}{a}+\mu}$ with $b>0$, $a>0$, $\mu>-1$, and $\nu>\max\{-b,-a\mu\}$ satisfies the following relation (which is nothing else but the Volterra equation~\eqref{eq:Volterra for Phi}):
\begin{align*}
&\Gamma\left(\frac{b}{a}  \right)E^{1+\frac{\nu-a}{b}}_{\frac{b}{a},\,\frac{\nu}{a}+\mu}\left( t^bz  \right)=1+az\intl_0^t  (t^a-s^a)^{\frac{b}{a}-1}t^{a-\nu}s^{\nu-1}\times\\
&
\times F_3\left(\frac{\nu}{a}-1,\frac{b}{a},1,\mu, \frac{b}{a}, 1-\left(\frac{s}{t}\right)^a,1-\left(\frac{t}{s}\right)^a\right)E^{1+\frac{\nu-a}{b}}_{\frac{b}{a},\,\frac{\nu}{a}+\mu}\left( s^bz  \right)\,ds,\quad t>0,\,\, z\in\cC.
\end{align*}
The special case of this relation for the classical Mittag-Leffler function is well-known and can be found, e.g. in Lemma~3.24 of~\cite{MR3244285}. 
\end{remark}

\section{Stochastic solutions with  stationary increments} \label{sec:fsm}

Some of the stochastic solutions,
that we derived,  (e.g., the ones in Theorem \ref{thm:hom}) were provided by randomly slowed-down / speeded-up L\'evy processes 
$(Y_{At^\beta})_{t\geq 0}$, where the positive random variable $A$ is independent of the L\'evy process $Y$. These processes may lack nice statistical properties which are appealing from a modeling point of view such as stationarity of the increments (which only holds for $\beta=1$) or self-similarity. Extending results of \cite{MR3513003} beyond the Gaussian 
case by the techniques explained in Example \ref{ex:fractional} above, we derive in this section stochastic solutions in terms of linear fractional 
stable motion. For the sake of exposition we restrict ourselves to one space dimension.

We first need to fix some notation. For the index of stability $\delta\in (0,2]$ and for 
the skewness parameter $\rho\in [-1,1]$ (with the restriction to the `symmetric case' $\rho=0$ for $\delta=1$), we consider the symbol
$-\psi_{\delta,\rho}$, where
$$
\psi_{\delta,\rho}(p)=|p|^\delta(1-i\rho\sign(p)\tan(\pi\delta/2)),\quad p\in \mathbb{R},
$$
covering the fractional Laplacian in space. The corresponding \emph{stable random
measure} $M_{\delta,\rho}$ is a $\sigma$-additive mapping from the Borel field $\mathcal{B}$ on $\mathbb{R}$ to the space of real-valued random variables which is \emph{randomly scattered} in the sense that 
$$
(M_{\delta,\rho}(A_1),\ldots, M_{\delta,\rho}(A_n))
$$
are independent, whenever $A_1,\ldots, A_n$ are pairwise disjoint, and such that $M_{\delta,\rho}(A)$ follows a stable law; precisely, 
$$
\eE[e^{ipM_{\delta,\rho}(A)}]=e^{-\Leb(A)\psi_{\delta,\rho}(p)},\qquad p\in \mathbb{R}, \quad A\in \mathcal{B}. 
$$
Here $\Leb$ denotes the Lebesgue measure on the real line. More details on stable random measures and integration with respect to them can be found in Chapter 3 of \cite{MR1280932}. The stable L\'evy motion corresponding to the characteristic exponent $\psi_{\delta,\rho}$ can be realized as 
$$
Y^{({\delta,\rho})}_t=M_{\delta,\rho}([0,t]),\quad t\geq 0.
$$
Note that, in the Gaussian case $\delta=2$, the normalization is chosen such that 
$
(\frac{1}{\sqrt{2}} Y^{({2,\rho})}_t)_{t\geq 0}
$
is a standard Brownian motion for any choice of $\rho\in [-1,1]$.

We may now consider a \emph{linear fractional L\'evy motion} of the form 
$$
Y^{(\delta,\rho,H)}_t=\frac{1}{K_{\delta,H}} \int_\mathbb{R} \left((t-x)_+^{H-1/\delta} -(-x)_+^{H-1/\delta}\right)\;M_{\delta,\rho}(dx),\quad t\geq 0
$$
for $H\in(0,1)\setminus \{1/\delta\}$, which contains the Mandelbrot-Van Ness representation for fractional Brownian motion as special case for $\delta=2$ (up to the factor $1/\sqrt{2}$ as explained above). Here, the normalizing constant is
$$
K_{\delta,H}=\left(\int_0^\infty \left|(1+x)^{H-1/\delta} -x^{H-1/\delta}\right|^{\delta} dx +\frac{1}{\delta H}\right)^{1/\delta}.
$$
By Proposition 7.4.2 in \cite{MR1280932}, linear fractional stable motion 
$(Y^{(\delta,\rho,H)}_t)_{t\geq 0}$ has stationary increments and is $H$-self-similar.
Moreover, the characteristic function of its one-dimensional marginals is given by 
\begin{equation}\label{eq:char_fsm1}
\varphi_{Y^{(\delta,\rho,H)}_t}(p)=\eE\left[e^{ipY^{(\delta,\rho,H)}_t}\right]=e^{-t^{\delta H}\psi_{\delta,\rho_0}(p)},\qquad t\geq 0,\quad p\in \mathbb{R},
\end{equation}
for
\begin{equation}\label{eq:char_fsm2}
 \rho_0=\left\{ \begin{array}{cl} \rho,& H>1/\delta \\ \rho \; \frac{\frac{1}{H\delta} - \int_0^\infty \left(x^{H-1/\delta}-(1+x)^{H-1/\delta} \right)^{\delta} dx}{\frac{1}{H\delta} + \int_0^\infty \left(x^{H-1/\delta}-(1+x)^{H-1/\delta} \right)^{\delta} dx}, & H<1/\delta \end{array} \right.,
\end{equation}
which can be obtained from Proposition 3.4.1 in \cite{MR1280932} by elementary computations.

Let us denote by $L_{\delta,\rho}$ the pseudo-differential operator associated to the symbol $-\psi_{\delta,\rho}$ via \eqref{eq:L}. In the case $\rho=0$ of the symmetric fractional Laplacian,
we also write $\Delta^{\delta/2}:=L_{\delta,0}$
\begin{theorem}\label{thm:flm}
 Suppose that $k$  is homogeneous of degree $\beta-1$ for some $\beta\in(0,2)$ and $k(1,\cdot)\in L^{1+\varepsilon}((0,1))$ for some $\varepsilon>0$,  and that 
 $x\mapsto \hat \Phi(-x)$ is completely monotone on $(0,\infty)$, where $\hat \Phi$ is defined 
 in Theorem \ref{thm:hom}. Then:
 \\[0.2cm]
 (i) Let  $\gamma\in(\beta,2]$ and $\delta\in [\gamma,2]\setminus \{\gamma/\beta\}$.
 If $A_{\gamma/\delta}$ is a nonnegative random variable with Laplace transform 
 $\hat\Phi(-(\cdot)^{\gamma/\delta})$ and $\big(Y^{(\delta,0,\beta/\gamma)}_t\big)_{t\geq 0}$ is a symmetric linear fractional stable motion independent of $A_{\gamma/\delta}$, then $\left(A_{\gamma/\delta}^{1/\delta} Y^{(\delta,0,\beta/\gamma)}_t\right)_{t\geq 0}$ provides a stochastic solution to 
 \begin{align*}
u(t,x)&=u_0(x)+\int_0^t k(t,s) \Delta^{\gamma/2} u(s,x)ds,\qquad t>0,\quad x\in\cR, \\
\lim_{t\searrow 0} u(t,x)&=u_0(x),\qquad  x\in\cR.\nonumber
\end{align*}
(ii) Let $\beta\neq1$, $\delta\in (\beta,2]\setminus\{1\}$, and $\rho\in [-1,1]$. If
$A$ is a nonnegative random variable with Laplace transform 
 $\hat\Phi(-(\cdot))$ and $\big(Y^{(\delta,\rho,\beta/\delta)}_t\big)_{t\geq 0}$ is a linear fractional stable motion independent of $A$, then $\left(A^{1/\delta} Y^{(\delta,\rho,\beta/\delta)}_t\right)_{t\geq 0}$ provides a stochastic solution to 
 \begin{align*}
u(t,x)&=u_0(x)+\int_0^t k(t,s) L_{\delta,\rho_0}u(s,x)ds,\qquad t>0,\quad x\in\cR, \\
\lim_{t\searrow 0} u(t,x)&=u_0(x),\qquad  x\in\cR,\nonumber
\end{align*}
for 
$$
\rho_0=\left\{
\begin{array}{ll}
{\rho,} & \beta\in(1,2),\\
\rho \; \frac{\frac{1}{\beta} - \int_0^\infty \left(x^{(\beta-1)/\delta}-(1+x)^{(\beta-1)/\delta} \right)^{\delta} dx}{\frac{1}{\beta} + \int_0^\infty \left(x^{(\beta-1)/\delta}-(1+x)^{(\beta-1)/\delta} \right)^{\delta} dx}, & \beta\in(0,1).
\end{array}\right.
$$
\end{theorem}
\begin{proof}
 (i) Let $H=\beta/\gamma$. In view of \eqref{eq:char_fsm1}--\eqref{eq:char_fsm2} and the independence assumption, we obtain
  \begin{eqnarray*}
 \eE\left[e^{ip A_{\gamma/\delta}^{1/\delta} Y^{(\delta,0,H)}_t}\right]&=&\eE\left[e^{-t^{\delta H}|p A_{\gamma/\delta}^{1/\delta} |^\delta}\right]= \eE\left[e^{-t^{\delta H}A_{\gamma/\delta}|p |^\delta }\right]=\hat \Phi(-(|p |^\delta t^{\delta H})^{\gamma/\delta})\\ &=& \hat \Phi(-|p|^\gamma t^\beta)
 =\Phi(t,-\psi_{\gamma,0}(p)),\qquad p\in \mathbb{R},\quad t\geq 0,
 \end{eqnarray*}
 applying Theorem \ref{thm:hom} for the last identity. The latter theorem also implies that 
 all assumptions of Theorem \ref{thm:generalTheorem} are satisfied.
 Thus, Theorem \ref{thm:generalTheorem}, (i), concludes the proof.
 \\[0.2cm]
 (ii) The proof is analogous to the one of part (i), %noting that $H:=\beta/\delta<1/\delta$ and
  making use of the computation 
 \begin{eqnarray*}
  \eE\left[e^{ip A^{1/\delta} Y^{(\delta,\rho,H)}_t}\right]&=&
  \eE\left[e^{-t^{\delta H} \psi_{\delta,\rho_0}(p A^{1/\delta})} \right]=\eE\left[e^{-t^{\beta}A \psi_{\delta,\rho_0}(p)} \right]=\hat \Phi(-t^\beta \psi_{\delta,\rho_0}(p)) \\
  &=& \Phi(t, -\psi_{\delta,\rho_0}(p)),\qquad p\in \mathbb{R},\quad t\geq 0.
 \end{eqnarray*}
\end{proof}
\begin{remark}
 (i) If $\delta=\gamma/\beta$ in the setting of Theorem \ref{thm:flm}, (i), then a stochastic solution is provided by the process $(\mathcal{A}^{1/\delta}_\beta Y^{(\delta,0)}_t)_{t\geq 0}$, where the symmetric 
 stable L\'evy motion $ Y^{(\delta,0)}$ (with characteristic exponent $\psi_{\delta,0}$) is independent of the positive random variable $\mathcal{A}_\beta$ with Laplace transform $\hat\Phi(-(\cdot)^{\beta})$. The proof remains valid without any changes.
 \\[0.2cm]
 (ii) The name `linear fractional L\'evy motion' usually refers to a larger family of processes 
 which have stationary increments, feature $H$-self-similarity and have stable laws, see Definition 7.4.1 in
 \cite{MR1280932}. We here chose the parametrization corresponding to the Mandelbrot-Van Ness representation of fractional Brownian motion. We note that other parameter choices work equally well, but -- except in the  Gaussian case -- may lead to different processes which only have identical one-dimensional marginal distributions.
\end{remark}

We conclude the paper by a summarizing example, combining some of the results of Sections \ref{sec:special case} and \ref{sec:fsm}.
\begin{example}
 Suppose $k$ is a Saigo-Maeda kernel
 $$
 k(t,s):=\frac{a}{\Gamma(b/a)}(t^a-s^a)^{\frac{b}{a}-1}t^{a-\nu}s^{\nu-1} F_3\left(\frac{\nu}{a}-1,\frac{b}{a},1,\mu, \frac{b}{a}, 1-\left(\frac{s}{t}\right)^a,1-\left(\frac{t}{s}\right)^a\right)
 $$
 for parameters
 $b\in(0,2)$, $a\geq b$, $\mu\geq \frac{b}{a}-1$, $\nu>\max\left\{a-b,-a\mu   \right\}$. 
 For $\gamma \in (b,2]$ and $\delta \in [\gamma,2]\setminus \{\gamma/b\}$ denote by $A_{\gamma/\delta,a,b,\mu,\nu}$ a random variable with Laplace transform given in terms of the three-parameter Mittag-Leffler function by  $$[0,\infty)\rightarrow \mathbb{R},\quad x\mapsto\Gamma\left(\frac{\nu}{a}+\mu\right) E_{\frac{b}{a},\frac{\nu}{a}+\mu}^{1+\frac{\nu-a}{b}}(-x^{\gamma/\delta}),$$ which is a completely monotone function 
 by Remark \ref{rem:CM-3parameterMLF}. Let 
 $(Y^{(\delta,0,b/\gamma)}_t)_{t\geq 0}$ be a symmetric linear fractional stable motion independent of $A_{\gamma/\delta,a,b,\mu,\nu}$. Then, for every initial condition $u_0\in S(\cRd)$, the function
 $$
 u(t,x)=\eE\left[u_0\left( x+A_{\gamma/\delta,a,b,\mu,\nu}^{1/\delta} Y^{(\delta,0,b/\gamma)}_t\right)\right],\qquad t\geq, x\in \mathbb{R}
 $$
 is a solution to
  \begin{align*}
u(t,x)&=u_0(x)+\int_0^t k(t,s) \Delta^{\gamma/2} u(s,x)ds,\qquad t>0,\quad x\in\cR, \\
\lim_{t\searrow 0} u(t,x)&=u_0(x),\qquad  x\in\cR.\nonumber
\end{align*}
 As $\delta$ is not a parameter of the equation, we, thus, obtain a whole family of stochastic representations in terms of $b/\gamma$-self-similar processes with stationary increments parametrized by the index of stability $\delta\in [\gamma,2]$. Here, $\gamma$ is the order of the space derivative
 and $b-1$ the degree of homogeneity of the Saigo-Maeda kernel. This example extends the results of 
 \cite{MR3513003} beyond the Gaussian case ($\delta=2$) and the time-fractional case of order $b\in(0,1]$ ($a=\nu=1$, $\mu=0$), cp. Example \ref{ex:saigo_maeda}, (ii).
  
\end{example}

%%%%%%%%%%%%%%%%%%%%%%%%%%%%%%%%%%%%%%%%%%%%%%%%%%%%%%%%%%%%%%%%%%%%%%%%%%%%%%%%%%%%%%%%%%%%%%%%%%%%%%%%%%%%%%%
%\newpage
%%%%%%%%%%%%%%%%%%%%%%%%%%%%%%%%%%%%%%%%%%%%%%%%%%%%%%%%%%%%%%%%%%%%%%%%%%%%%%%%%%%%%%%%%%%%%%%%%%%%%%%%%%%%%%%

%%%%%%%%%%%%%%%%%%%%%%%%%%%%%%%%%%%%%%%%%%%%%%%%%%%%%%%%%%%%%%%%%%%%%%%%%%%%%%%%%%%%%%%%%%%%%%%%%%%%%%%%%%%%%%%%%%

%\bibliographystyle{abbrv}
%
%\bibliography{BeBu_2020}

\begin{thebibliography}{10}

\bibitem{MR1703898}
S.~Albeverio, A.~Khrennikov, and O.~Smolyanov.
\newblock The probabilistic {F}eynman-{K}ac formula for an infinite-dimensional
  {S}chr\"{o}dinger equation with exponential and singular potentials.
\newblock {\em Potential Anal.}, 11(2):157--181, 1999.

\bibitem{An}
J.~An, E.~Van~Hese, and M.~Baes.
\newblock Phase-space consistency of stellar dynamical models determined by
  separable augmented densities.
\newblock {\em Monthly Notices of the Royal Astronomical Society},
  422(1):652--664, 2012.

\bibitem{MR2512800}
D.~Applebaum.
\newblock {\em L\'evy processes and stochastic calculus}, volume 116 of {\em
  Cambridge Studies in Advanced Mathematics}.
\newblock Cambridge University Press, Cambridge, 2 edition, 2009.

\bibitem{MR1874479}
B.~Baeumer and M.~M. Meerschaert.
\newblock Stochastic solutions for fractional {C}auchy problems.
\newblock {\em Fract. Calc. Appl. Anal.}, 4(4):481--500, 2001.

\bibitem{doi:10.1080/17442508.2019.1641093}
W.~Bock, S.~Desmettre, and J.~L. da~Silva.
\newblock Integral representation of generalized grey {B}rownian motion.
\newblock {\em Stochastics}, pages 1--14, 2019.

\bibitem{ND}
Y.~A. Butko.
\newblock The {F}eynman-{K}ac-{I}to formula for an infinite-dimensional
  {S}chr\"{o}dinger equation with scalar and vector potentials.
\newblock {\em Nelin. Dinam.}, 2(1):75--87, 2006.

\bibitem{MR3280006}
A.~G. Cherstvy, A.~V. Chechkin, and R.~Metzler.
\newblock Ageing and confinement in non-ergodic heterogeneous diffusion
  processes.
\newblock {\em J. Phys. A}, 47(48):485002, 18, 2014.

\bibitem{PhysRevLett.113.098302}
M.~V. Chubynsky and G.~W. Slater.
\newblock Diffusing diffusivity: A model for anomalous, yet {B}rownian,
  diffusion.
\newblock {\em Phys. Rev. Lett.}, 113:098302, Aug 2014.

\bibitem{MR3733422}
J.~L. da~Silva and M.~Erraoui.
\newblock Existence and upper bound for the density of solutions of stochastic
  differential equations driven by generalized grey noise.
\newblock {\em Stochastics}, 89(6-7):1116--1126, 2017.

\bibitem{MR3854539}
J.~L. da~Silva and M.~Erraoui.
\newblock Singularity of generalized grey {B}rownian motions with different
  parameters.
\newblock {\em Stoch. Anal. Appl.}, 36(4):726--732, 2018.

\bibitem{MR3956716}
J.~L. da~Silva and L.~Streit.
\newblock Structure factors for generalized grey {B}rowinian motion.
\newblock {\em Fract. Calc. Appl. Anal.}, 22(2):396--411, 2019.

\bibitem{MR574173}
H.~Doss.
\newblock Sur une r\'{e}solution stochastique de l'\'{e}quation de
  {S}chr\"{o}dinger \`a coefficients analytiques.
\newblock {\em Comm. Math. Phys.}, 73(3):247--264, 1980.

\bibitem{MR3903618}
M.~D'Ovidio, S.~Vitali, V.~Sposini, O.~Sliusarenko, P.~Paradisi, G.~Castellani,
  and G.~Pagnini.
\newblock Centre-of-mass like superposition of {O}rnstein-{U}hlenbeck
  processes: a pathway to non-autonomous stochastic differential equations and
  to fractional diffusion.
\newblock {\em Fract. Calc. Appl. Anal.}, 21(5):1420--1435, 2018.

\bibitem{MR3400947}
R.~Garra, E.~Orsingher, and F.~Polito.
\newblock Fractional diffusions with time-varying coefficients.
\newblock {\em J. Math. Phys.}, 56(9):093301, 17, 2015.

\bibitem{MR2551283}
G.~Germano, M.~Politi, E.~Scalas, and R.~L. Schilling.
\newblock Stochastic calculus for uncoupled continuous-time random walks.
\newblock {\em Phys. Rev. E (3)}, 79(6):066102, 12, 2009.

\bibitem{MR3244285}
R.~Gorenflo, A.~A. Kilbas, F.~Mainardi, and S.~V. Rogosin.
\newblock {\em Mittag-{L}effler functions, related topics and applications}.
\newblock Springer Monographs in Mathematics. Springer, Heidelberg, 2014.

\bibitem{Gorska2020}
K.~G\'{o}rka, A.~Horzela, A.~Lattanzi, and T.~K. Pog\'{a}ny.
\newblock On complete monotonicity of three parameter {M}ittag-{L}effler
  function.
\newblock {\em Applicable Analysis and Discrete Mathematics}, 2021.

\bibitem{MR3464056}
M.~Grothaus and F.~Jahnert.
\newblock Mittag-{L}effler analysis {II}: {A}pplication to the fractional heat
  equation.
\newblock {\em J. Funct. Anal.}, 270(7):2732--2768, 2016.

\bibitem{MR2800586}
H.~J. Haubold, A.~M. Mathai, and R.~K. Saxena.
\newblock Mittag-{L}effler functions and their applications.
\newblock {\em J. Appl. Math.}, 298628, 2011.

\bibitem{MR1873235}
N.~Jacob.
\newblock {\em Pseudo differential operators and {M}arkov processes. {V}ol.
  {I}}.
\newblock Imperial College Press, London, 2001.
\newblock Fourier analysis and semigroups.

\bibitem{Jain2017}
R.~Jain and K.~L. Sebastian.
\newblock Diffusing diffusivity: a new derivation and comparison with
  simulations.
\newblock {\em Journal of Chemical Sciences}, 129(7):929--937, Jul 2017.

\bibitem{MR884309}
J.~Klafter, A.~Blumen, and M.~F. Shlesinger.
\newblock Stochastic pathway to anomalous diffusion.
\newblock {\em Phys. Rev. A (3)}, 35(7):3081--3085, 1987.

\bibitem{MR2766141}
V.~N. Kolokoltsov.
\newblock Generalized continuous-time random walks, subordination by hitting
  times, and fractional dynamics.
\newblock {\em Teor. Veroyatn. Primen.}, 53(4):684--703, 2008.

\bibitem{MR3987876}
V.~N. Kolokoltsov.
\newblock The probabilistic point of view on the generalized fractional partial
  differential equations.
\newblock {\em Fract. Calc. Appl. Anal.}, 22(3):543--600, 2019.

\bibitem{MR2848339}
J.~L\H{o}rinczi, F.~Hiroshima, and V.~Betz.
\newblock {\em Feynman-{K}ac-type theorems and {G}ibbs measures on path space},
  volume~34 of {\em De Gruyter Studies in Mathematics}.
\newblock Walter de Gruyter \& Co., Berlin, 2011.
\newblock With applications to rigorous quantum field theory.

\bibitem{MR2442372}
M.~M. Meerschaert and H.-P. Scheffler.
\newblock Triangular array limits for continuous time random walks.
\newblock {\em Stochastic Process. Appl.}, 118(9):1606--1633, 2008.

\bibitem{MR2884383}
M.~M. Meerschaert and A.~Sikorskii.
\newblock {\em Stochastic models for fractional calculus}, volume~43 of {\em De
  Gruyter Studies in Mathematics}.
\newblock Walter de Gruyter \& Co., Berlin, 2012.

\bibitem{MR1809268}
R.~Metzler and J.~Klafter.
\newblock The random walk's guide to anomalous diffusion: a fractional dynamics
  approach.
\newblock {\em Phys. Rep.}, 339(1):77, 2000.

\bibitem{MR172344}
E.~W. Montroll and G.~H. Weiss.
\newblock Random walks on lattices. {II}.
\newblock {\em J. Mathematical Phys.}, 6:167--181, 1965.

\bibitem{MR2501791}
A.~Mura and F.~Mainardi.
\newblock A class of self-similar stochastic processes with stationary
  increments to model anomalous diffusion in physics.
\newblock {\em Integral Transforms Spec. Funct.}, 20(3-4):185--198, 2009.

\bibitem{MR2430462}
A.~Mura and G.~Pagnini.
\newblock Characterizations and simulations of a class of stochastic processes
  to model anomalous diffusion.
\newblock {\em J. Phys. A}, 41(28):285003, 22, 2008.

\bibitem{MR2588003}
A.~Mura, M.~S. Taqqu, and F.~Mainardi.
\newblock Non-{M}arkovian diffusion equations and processes: analysis and
  simulations.
\newblock {\em Phys. A}, 387(21):5033--5064, 2008.

\bibitem{MR3586912}
G.~Pagnini.
\newblock Fractional kinetics in random/complex media.
\newblock In {\em Handbook of fractional calculus with applications. {V}ol. 5},
  pages 183--205. De Gruyter, Berlin, 2019.

\bibitem{MR3513003}
G.~Pagnini and P.~Paradisi.
\newblock A stochastic solution with {G}aussian stationary increments of the
  symmetric space-time fractional diffusion equation.
\newblock {\em Fract. Calc. Appl. Anal.}, 19(2):408--440, 2016.

\bibitem{MR0027375}
H.~Pollard.
\newblock The completely monotonic character of the {M}ittag-{L}effler function
  ${E}_a(-x)$.
\newblock {\em Bull. Amer. Math. Soc.}, 54:1115--1116, 1948.

\bibitem{MR1280932}
G.~Samorodnitsky and M.~S. Taqqu.
\newblock {\em Stable non-{G}aussian random processes. Stochastic models with
  infinite variance}.
\newblock Stochastic Modeling. Chapman \& Hall, New York, 1994.

\bibitem{Saxena}
R.~K. Saxena.
\newblock Chapter 3. {F}ractional {C}alculus.
\newblock In {\em Lecture {N}otes of the 5th {S}.{E}.{R}.{C}. {S}chool on
  {S}pecial {F}unctions and {F}unctions of {M}atrix {A}rgument: {R}ecent
  {A}dvances and {A}pplications in {S}tochastic {P}rocesses, {S}tatistics,
  {W}avelet {A}nalysis and {A}strophysics}, pages 79--108. Centre for
  {M}athematical {S}ciences. {P}ala {C}ampus, {S}t. {T}homas {C}ollege,
  {A}runapuram {P}. {O}., {P}ala, {K}erala 686574, {I}ndia, 2007.

\bibitem{MR2097999}
E.~Scalas, R.~Gorenflo, and F.~Mainardi.
\newblock Uncoupled continuous-time random walks: solution and limiting
  behavior of the master equation.
\newblock {\em Phys. Rev. E (3)}, 69(1):011107, 8, 2004.

\bibitem{MR2978140}
R.~L. Schilling, R.~Song, and Z.~Vondra\v{c}ek.
\newblock {\em Bernstein functions}, volume~37 of {\em De Gruyter Studies in
  Mathematics}.
\newblock Walter de Gruyter \& Co., Berlin, second edition, 2012.
\newblock Theory and applications.

\bibitem{MR1124240}
W.~R. Schneider.
\newblock Grey noise.
\newblock In {\em Stochastic processes, physics and geometry ({A}scona and
  {L}ocarno, 1988)}, pages 676--681. World Sci. Publ., Teaneck, NJ, 1990.

\bibitem{MR1190506}
W.~R. Schneider.
\newblock Grey noise.
\newblock In {\em Ideas and methods in mathematical analysis, stochastics, and
  applications ({O}slo, 1988)}, pages 261--282. Cambridge Univ. Press,
  Cambridge, 1992.

\bibitem{MR544188}
B.~Simon.
\newblock {\em Functional integration and quantum physics}, volume~86 of {\em
  Pure and Applied Mathematics}.
\newblock Academic Press, Inc. [Harcourt Brace Jovanovich, Publishers], New
  York-London, 1979.

\bibitem{MR3916448}
O.~Y. Sliusarenko, S.~Vitali, V.~Sposini, P.~Paradisi, A.~Chechkin,
  G.~Castellani, and G.~Pagnini.
\newblock Finite-energy {L}\'{e}vy-type motion through heterogeneous ensemble
  of {B}rownian particles.
\newblock {\em J. Phys. A}, 52(9):095601, 27, 2019.

\bibitem{Sposini_2018}
V.~Sposini, A.~V. Chechkin, F.~Seno, G.~Pagnini, and R.~Metzler.
\newblock Random diffusivity from stochastic equations: comparison of two
  models for {B}rownian yet non-{G}aussian diffusion.
\newblock {\em New Journal of Physics}, 20(4):043044, apr 2018.

\end{thebibliography}

\end{document}